\documentclass[leqno, 12pt]{article}
\usepackage{amsmath,amsfonts,amsthm,amssymb,indentfirst}
\usepackage{xy} \xyoption{all}
\usepackage[colorlinks]{hyperref}
\hypersetup{linkcolor=blue, urlcolor=blue, citecolor=red}

\newtheorem{theorem}{Theorem}[section]

\newtheorem{corollary}[theorem]{Corollary}

\newtheorem{lemma}[theorem]{Lemma}

\newtheorem{proposition}[theorem]{Proposition}

\theoremstyle{definition}
\newtheorem{definition}[theorem]{Definition}

\newtheorem{example}[theorem]{Example}
\newtheorem{remark}[theorem]{Remark}

\newcommand{\Z}{\mathbb{Z}}

\newcommand{\M}{\mathbb{M}}

\begin{document}

\title{Products and Intersections of Prime-Power Ideals in Leavitt Path Algebras}

\author{Zachary Mesyan and Kulumani M.\ Rangaswamy}

\date{}

\maketitle

\begin{abstract}
We continue a very fruitful line of inquiry into the multiplicative ideal theory of an arbitrary Leavitt path algebra $L$. Specifically, we show that factorizations of an ideal in $L$ into irredundant products or intersections of finitely many prime-power ideals are unique, provided that the ideals involved are powers of distinct prime ideals. We also characterize the completely irreducible ideals in $L$, which turn out to be prime-power ideals of a special type, as well as ideals that can be factored into products or intersections of finitely many completely irreducible ideals.

\medskip


\end{abstract}


\section{Introduction}

The ideal theory of the Leavitt path algebra $L=L_K(E)$ of a directed graph $E$ over a field $K$ has been an active area of research in recent years. In particular, a number of papers have been devoted to characterizing special types of ideals of $L$ in terms of graphical properties of $E$, and to describing the ideals of $L$ that can be factored into products of ideals of these types. More specifically prime, primitive, semiprime, and irreducible ideals have received such treatment in the literature--see \cite{ABR, AMR, ARRS, EEKRR, R-1, R-2}. An interesting feature of Leavitt path algebras is that, while they are highly noncommutative, multiplication of their ideals is commutative, and further, their ideals share a number of properties with ideals in various commutative rings, such as Dedekind domains (where ideals are projective), B\'{e}zout rings (where finitely-generated ideals are principal), arithmetical rings (where the ideal lattices are distributive), and Pr\"{u}fer domains (where the ideal lattices have special properties). Taken together, the aforementioned literature has produced a rich multiplicative ideal theory for Leavitt path algebras. Our goal in this paper is to contribute to this line of inquiry by examining ideals in Leavitt path algebras that can be factored into finite products and intersections of prime-power ideals and variants thereof.

After recalling the relevant background material (Section~\ref{prelim-section}),  in Section~\ref{prime-prod-section} we show that in any Leavitt path algebra $L$, factorizations of an ideal into irredundant products of prime-power ideals are unique, up to reordering the factors, provided that they are powers of different prime ideals (Theorem~\ref{Uniqueness}). A consequence of this is that factorizations of an ideal of $L$ into irredundant products of prime ideals are unique, up to reordering the factors (Corollary~\ref{PrimeUniqueness}). Then, in Section~\ref{prime-intersect-section}, we deduce that decompositions of an ideal of $L$ into irredundant intersections of finitely many prime-power ideals are likewise unique (Theorem~\ref{uniqueness-intersect}), using the fact that the product of any finite collection of powers of distinct prime ideals in $L$ is equal to its intersection (Proposition~\ref{prime-power-intersect}).

In Section~\ref{comp-irred-section} we classify the completely irreducible ideals in an arbitrary Leavitt path algebra, showing that they are prime-power ideals of a special sort. Recall that a proper ideal $I$ of a ring $R$ is \textit{irreducible} if $I$ is not the intersection of any two ideals of $R$ properly containing $I$, and $I$ is \textit{completely irreducible} if $I$ is not the intersection of any set of ideals properly containing $I$. For example, every prime ideal in any ring is clearly irreducible, but the zero ideal in $\Z$ is not completely irreducible despite being prime. Moreover, every maximal ideal in any ring is trivially completely irreducible. In the case of commutative rings, these ideals and variants thereof were investigated in a series of papers by Fuchs, Heinzer, Mosteig, and Olberding \cite{FHO-0, FHO, FM,HO}, leading to interesting factorization theorems. 

Irreducible ideals in Leavitt path algebras were described in \cite{R-2, AMR}, but completely irreducible ideals have not received much attention before. In Theorem~\ref{comp.Irreducible}, we classify the completely irreducible ideals of an arbitrary Leavitt path algebra $L$. Specifically, such ideals must either be powers of non-graded prime ideals, or be graded prime ideals of a special type. We then characterize the Leavitt path algebras $L$ where every proper ideal is completely irreducible (Theorem~\ref{everyidealcompirred}), where every completely irreducible ideal is graded (Proposition~\ref{Every comp.irred.graded}), and where every irreducible ideal is completely irreducible (Proposition~\ref{Irredcible = Comp Irred}). In each case we give equivalent conditions on $L$, as well as equivalent conditions on the graph $E$. For example, every proper ideal in $L$ completely irreducible exactly when every ideal of $L$ is graded, and the ideals of $L$ are well-ordered under set inclusion.

In the final Section~\ref{prod-comp-irred-section} we characterize the ideals in an arbitrary Leavitt path algebra that can be factored into products or intersections of finitely many completely irreducible ideals (Theorem~\ref{Intersection of completely irreducibles}), as well as the Leavitt path algebras where every proper ideal is the product or intersection of finitely many completely irreducible ideals (Theorem~\ref{all-prod-irred}).
Note that every proper ideal in any ring $R$ is the intersection of all the completely irreducible ideals containing it (see Proposition~\ref{Everyone is int comp.irr.}), and so only intersections of finitely many completely irreducible ideals are of interest. 

Graphical examples are constructed illustrating many of our results.

\section{Preliminaries}\label{prelim-section}

We begin with a review of the basic concepts and facts required. General Leavitt path algebra terminology and results can be found in~\cite{AAS}. For the reader familiar with these algebras, we suggest skipping the rest of this section and referring back as needed.

\subsection{Graphs} \label{graphs-section}

A \textit{directed graph} $E=(E^{0},E^{1},r,s)$ consists of two sets $E^{0}$ and $E^{1}$, where $E^0 \neq \emptyset$, together with functions $r,s:E^{1}\rightarrow E^{0}$, called \emph{range} and \emph{source}, respectively. The elements of $E^{0}$ are called \textit{vertices}, and the elements of $E^{1}$ are called \textit{edges}. We shall refer to directed graphs as simply ``graphs" from now on.

Let $E=(E^{0},E^{1},r,s)$ be a graph. A \emph{path} $\mu$ in $E$ is a finite sequence $e_1\cdots e_n$ of edges $e_1,\dots, e_n \in E^1$ such that $r(e_i)=s(e_{i+1})$ for $i\in \{1,\dots,n-1\}$. Here we define $s(\mu):=s(e_1)$ to be the \emph{source} of $\mu$, $r(\mu):=r(e_n)$ to be the \emph{range} of $\mu$, and $|\mu|:=n$ to be the \emph{length} of $\mu$. We view the elements of $E^0$ as paths of length $0$, and extend $s$ and $r$ to $E^0$ via $s(v) = v = r(v)$ for all $v\in E^0$. The set of all vertices on a path $\mu$ is denoted by $\{\mu^{0}\}$. A path $\mu = e_1\cdots e_{n}$ is \textit{closed} if $r(e_{n}) = s(e_{1})$. A closed path $\mu = e_1\cdots e_{n}$ is a \textit{cycle} if $s(e_{i})\neq s(e_{j})$ for all $i\neq j$. An \textit{exit} for a path $\mu = e_1\cdots e_{n}$ is an edge $f \in E^{1} \setminus \{e_1, \dots, e_n\}$ that satisfies $s(f) = s(e_{i})$ for some $i$. The graph $E$ is said to satisfy \textit{condition (L)} if every cycle in $E$ has an exit, and to satisfy \textit{condition (K)} if any vertex on a closed path $\mu$ is also the source of a closed path different from $\mu$ (i.e., one possessing a different set of edges). A cycle $\mu$ in $E$ is said to be \textit{without (K)} if no vertex along $\mu$ is the source of a different cycle in $E$.

Given a vertex $v \in E^{0}$, we say that $v$ is a \textit{sink} if $s^{-1}(v) = \emptyset$, that $v$ is \textit{regular} if $s^{-1}(v)$ is finite but nonempty, and that $v$ is an \textit{infinite emitter} if $s^{-1}(v)$ is infinite. A graph without infinite emitters is said to be
\textit{row-finite}. If $u,v \in E^0$, and there is a path $\mu$ in $E$ satisfying $s(\mu) = u$ and $r(\mu)=v$, then we write $u \geq v$. Given a vertex $v \in E^0$, we set $M(v) = \{w \in E^0 \mid w \geq v\}$. A nonempty subset $H$ of $E^{0}$ is said to be \textit{downward directed} if for any $u,v\in H$, there exists $w\in H$ such that $u\geq w$ and $v\geq w$. A subset $H$ of $E^0$ is \emph{hereditary} if whenever $u \in H$ and $u \geq v$ for some $v \in E^0$, then $v \in H$. Also $H \subseteq E^0$ is \emph{saturated} if $r(s^{-1}(v)) \subseteq H$ implies that $v \in H$ for any regular $v \in E^0$.  A nonempty subset $M$ of $E^0$ is a \emph{maximal tail} if it satisfies the following three conditions.

(MT1) If $v\in M$ and $u\in E^0$ are such that $u \geq v$, then $u\in M$.

(MT2) For every regular $v \in M$ there exists $e \in E^1$ such that $s(e)=v$ and $r(e) \in M$. 

(MT3) The set $M$ is downward directed.\\
For any $H \subseteq E^0$ it is easy to see that $H$ is hereditary if and only if $M=E^0\setminus H$ satisfies (MT1), and $H$ is saturated if and only if $M=E^0\setminus H$ satisfies (MT2).

\subsection{Leavitt Path Algebras}

Given a graph $E$ and a field $K$, the \textit{Leavitt path $K$-algebra} $L_K(E)$ \textit{of $E$} is the $K$-algebra generated by the set $\{v \mid v\in E^{0}\} \cup \{e,e^* \mid e\in E^{1}\}$, subject to the following relations:

\smallskip

{(V)} \ \ \ \  $vw = \delta_{v,w}v$ for all $v,w\in E^{0}$,

{(E1)} \ \ \ $s(e)e=er(e)=e$ for all $e\in E^{1}$,

{(E2)} \ \ \ $r(e)e^*=e^*s(e)=e^*$ for all $e\in E^{1}$,

{(CK1)} \ $e^*f=\delta _{e,f}r(e)$ for all $e,f\in E^{1}$, and

{(CK2)} \  $v=\sum_{e\in s^{-1}(v)} ee^*$ for all regular $v\in E^{0}$.   

\smallskip

\noindent Throughout this article, $K$ will denote an arbitrary field, $E$ will denote an arbitrary graph, and $L_K(E)$ will often be denoted simply by $L$. 

For all $v \in E^0$ we define $v^*:=v$, and for all paths $\mu  = e_1 \cdots e_n$ ($e_1, \dots, e_n \in E^1$) we set $\mu^*:=e_n^* \cdots e_1^*$, $r(\mu^*):=s(\mu)$, and $s(\mu^*):=r(\mu)$. It is easy to see that every element of $L_K(E)$ can be expressed in the form $\sum_{i=1}^n a_i\mu_i\nu_i^*$ for some $a_i \in K$ and paths $\mu_i,\nu_i$. We also note that while $L_K(E)$ is generally not unital, it has \textit{local units}. That is, for any finite subset $\{r_1, \dots, r_n\}$ of $L_K(E)$ there is an idempotent $u\in L_K(E)$ (which can be taken to be a sum of vertices in $E^{0}$) such that $ur_i=r_i=r_iu$ for all $1\leq i \leq n$.  

We also recall that every Leavitt path algebra $L_{K}(E)$ is $\mathbb{Z}$-graded (where $\mathbb{Z}$ denotes the group of  integers). Specifically, $L_{K} (E)=\bigoplus_{n\in\mathbb{Z}} L_{n}$, where 
\[
L_{n}=\bigg\{\sum_i a_{i}\mu_{i}\nu_{i}^*\in L_{K}(E) \Bigm| n = |\mu_{i}|-|\nu_{i}|\bigg\}.
\] 
Here the \emph{homogeneous components} $L_{n}$ are abelian subgroups satisfying $L_{m}L_{n}\subseteq L_{m+n}$ for all $m, n\in \mathbb{Z}$.  An ideal $I$ of $L_{K}(E)$ is said to be \emph{graded} if $I = \bigoplus_{n\in\mathbb{Z}} (I\cap L_{n})$. Equivalently, $I$ is \emph{graded} if $r_{i_{j}}\in I$ for all $j \in \{1,\dots, m\}$, whenever $r_{i_{1}}+\cdots+r_{i_{m}} \in I$, with $i_1, \dots, i_m$ distinct and $r_{i_{j}}\in L_{i_{j}}$ for each $j$. 

\subsection{Ideals in Leavitt Path Algebras} \label{LPAidealSect}

Next we record various results from the literature and basic observations about ideals in Leavitt path algebras that will be used frequently in the paper.

Given a graph $E$, a \emph{breaking vertex} of a hereditary saturated subset $H$ of  $E^0$ is an infinite emitter $v\in E^{0}\backslash H$ with the property that $0<|s^{-1}(v)\cap r^{-1}(E^{0}\backslash H)|<\aleph_0$. The set of all breaking vertices of $H$ is denoted by $B_{H}$. For each $v\in B_{H}$, we set $v^{H} := v-\sum_{s(e)=v, \, r(e) \notin H}ee^*$. Given a hereditary saturated $H \subseteq E^0$ and $S\subseteq B_{H}$, we say that $(H,S)$ is an \emph{admissible pair}, and the ideal of $L_K(E)$ generated by $H\cup \{v^{H} \mid v\in S\}$ is denoted by $I(H,S)$. The graded ideals of $L_{K}(E)$ are precisely the ideals of the form $I(H,S)$ \cite[Theorem 2.5.8]{AAS}. Moreover, setting $(H_1,S_1) \leq (H_2,S_2)$ whenever $H_1 \subseteq H_2$ and $S_1 \subseteq H_2 \cup S_2$, defines a partial order on the set of all admissible pairs of $L_K(E)$. The map $(H,S) \mapsto I(H,S)$ gives a one-to-one order-preserving correspondence between the partially ordered set of admissible pairs and the set of all graded ideals of $L_K(E)$, ordered by inclusion. Finally, for any admissible pair $(H,S)$ we have $L_{K}(E)/I(H,S)\cong L_{K} (E\backslash(H,S))$, via a graded isomorphism (i.e., one that takes each homogeneous component in one ring to the corresponding homogeneous component in the other ring) \cite[Theorem 2.4.15]{AAS}. Here $E\backslash(H,S)$ is a \emph{quotient graph of} $E$, where \[(E\backslash(H,S))^{0}=(E^{0}\backslash H)\cup\{v' \mid v\in B_{H} \backslash S\}\]
and
\[(E\backslash(H,S))^{1}=\{e\in E^{1} \mid r(e)\notin H\}\cup\{e' \mid e\in E^{1} \text{ with } r(e)\in B_{H}\backslash S\},\]
and $r,s$ are extended to $E\backslash(H,S)$ by setting $s(e^{\prime}) = s(e)$ and $r(e')=r(e)'$. (We note that $(E\backslash(H,B_{H}))^{0} = E^{0} \backslash H$ for any hereditary saturated $H$.) More specifically, the aforementioned isomorphism $L_{K}(E)/I(H,S)\cong L_{K} (E\backslash(H,S))$ preserves the vertices in $E^{0}\backslash H$ and edges $e \in E^1$ whose ranges are not in $H$, along with the corresponding ghost edges $e^*$.

\begin{theorem}\label{arbitrary ideals}
Let $L=L_K(E)$ be a Leavitt path algebra, and let $I$ be an ideal of $L$, with $H=I\cap E^{0}$ and $S=\{v\in B_{H} \mid v^{H}\in I\}$. 
\begin{enumerate}
\item[$(1)$] \cite[Theorem 4]{R_0} \ $I=I(H,S)+\sum_{i\in Y} \langle f_{i}(c_{i})\rangle$ where $Y$ is a possibly empty index set; each $c_{i}$ is a cycle without exits in $E\backslash(H,S)$; and each $f_{i}(x)\in K[x]$ is a polynomial with a nonzero constant term, which is of smallest degree such that $f_{i}(c_{i})\in I$.

\item[$(2)$] \cite[Proposition 2.4.7]{AAS} \ Using the notation of $\, (1)$, we have $I/I(H,S)=\bigoplus_{i\in Y} \langle f_{i}(c_{i})\rangle\subseteq \bigoplus_{i\in Y} \langle \{c_{i}^{0}\} \rangle$, where $ \langle \{c_{i}^{0}\} \rangle$ is the ideal of $L/I(H,S)$ generated by $\, \{c_{i}^{0}\}$.

\item[$(3)$] \cite[Corollary 2.4.16]{AAS} \ $I(H,S) \cap E^{0} = H = \langle H \rangle \cap E^{0}$, where $\langle H \rangle$ is the ideal of $L$ generated by $H$.

\item[$(4)$] \cite[Corollary 2.8.17]{AAS} \ For any ideal $J$ of $L$, we have $IJ=JI$.
\end{enumerate}
\end{theorem}

For an ideal $I$ of $L$, expressed as in Theorem \ref{arbitrary ideals}(1), we refer to $I(H,S)$, also denoted $\mathrm{gr}(I)$, as the \textit{graded part} of $I$.

\begin{lemma} \label{Product=Intersection} \cite[Lemma 3.1]{R-2}
Let $I$ be a graded ideal of a Leavitt path algebra $L$.
\begin{enumerate}
\item[$(1)$] $IJ=I\cap J$ for any ideal $J$ of $L$. In particular, for any ideal $J \subseteq I$, we have $IJ=J$.

\item[$(2)$] Let $I_{1},\dots, I_{n}$ be ideals of $L$, for some positive integer $n$. Then $I=I_{1}\cdots I_{n}$ if and only if $I=I_{1}\cap\cdots\cap I_{n}$.
\end{enumerate}
\end{lemma}

\begin{lemma}\label{Product} 
Let $L=L_K(E)$ be a Leavitt path algebra, let $c$ and $d$ be distinct cycles in $E$ (that is, $c$ and $d$ have distinct sets of edges) without exits, and let $f(x), g(x)\in K[x]$ be polynomials with nonzero constant terms.
\begin{enumerate}
\item[$(1)$] \cite[Lemma 3.3]{R-2} \ $\langle f(c) \rangle \langle g(c) \rangle = \langle f(c)g(c) \rangle$.  

\item[$(2)$] \cite[Lemma 2.2(2)]{AMR} \ $\langle f(c)\rangle \langle g(d)\rangle = \{0\}$.
\end{enumerate}
\end{lemma}

\begin{lemma} \label{cycle-ideal-lemma} 
Let $L=L_K(E)$ be a Leavitt path algebra and $c$ a cycle without exits in $E$.
\begin{enumerate}
\item[$(1)$] \cite[Lemma 2.7.1]{AAS} \ $\langle \{c^{0}\}\rangle \cong \M_{Y}(K[x,x^{-1}])$ for some index set $Y$.

\item[$(2)$] \cite[Lemma 3.5]{R-1} \ If $E^0$ is downward directed, and $I$ is a nonzero ideal of $L$ containing no vertices, then $I=\langle f(c) \rangle \subseteq \langle \{c^{0}\}\rangle$, for some $f(x)\in K[x]$ with a nonzero constant term.

\item[$(3)$] If $E^0$ is downward directed, and $I$ is an ideal of $L$ containing a vertex, then $\langle \{c^{0}\}\rangle \subseteq I$.
\end{enumerate}
\end{lemma}

\begin{proof}[Proof of (3).]
First note that $E^0$ being downward directed, and $c$ having no exits, means that $u \geq s(c)$ for all $u \in E^0$. So if an ideal $I$ of $L$ contains a vertex $u$, then $I\cap E^{0}$ being hereditary (see Theorem~\ref{arbitrary ideals}) implies that $s(c) \in I$. It follows that $\langle \{c^{0}\}\rangle \subseteq I$.
\end{proof}

\begin{remark} \label{cycle-ideal-remark} 
We note, for future reference, that in the isomorphism $\langle \{c^{0}\}\rangle \cong \M_{Y}(K[x,x^{-1}])$ constructed in the proof of \cite[Lemma 2.7.1]{AAS}, $c \in \langle \{c^{0}\}\rangle$ is sent to a matrix in $\M_{Y}(K[x,x^{-1}])$ which has $x$ as one of the entries and zeros elsewhere.
\end{remark}

\begin{theorem} \label{specificprime}  \cite[Theorem 3.12]{R-1}
Let $L=L_K(E)$ be a Leavitt path algebra, let $I$ be a proper ideal if $L$, and let $H = I \cap E^0$. Then $I$ is a prime ideal if and only if $I$ satisfies one of the following conditions.
\begin{enumerate}
\item[$(1)$] $I = I(H, B_{H})$, and $E^0\setminus H$ is downward directed. 
\item[$(2)$] $I = I(H, B_{H}\setminus \{u\})$ for some $u \in B_H$, and $E^0\setminus H = M(u)$.
\item[$(3)$] $I = I(H, B_{H}) + \langle f(c) \rangle$ where $c$ is a cycle without (K), $E^0\setminus H = M(s(c))$, and $f(x) \in K[x, x^{-1}]$ is an irreducible polynomial.
\end{enumerate}
\end{theorem}

\begin{lemma} \label{Properties of primes} \cite[Corollary 4.5]{R-2}
Let $L$ be a Leavitt path algebra, let $P$ be a prime ideal of $L$, and let $I$ be an ideal of $L$. If $P \subsetneq I$, then $P=IP$.
\end{lemma}

\begin{remark}
An interesting consequence of Lemma~\ref{Properties of primes} is that if $P \subsetneq I$, for some prime ideal $P$ and ideal $I$ of $L$, then $P \subsetneq I^n$ for all positive integers $n$. (For, $P=IP$ implies that $P=I^nP \subseteq I^n$, and the inclusion must be proper, since otherwise we would have $P=I$.)
\end{remark}

\begin{proposition}\label{quasi-primaryequivalent} \cite[Proposition 3.2]{AMR}
Let $L=L_K(E)$ be a Leavitt path algebra, let $I$ be a proper ideal of $L$, and write $\,\mathrm{gr}(I) = I(H,S)$. Then the following are equivalent.
\begin{enumerate}
\item[$(1)$] $I$ is irreducible.

\item[$(2)$] $I$ is a prime-power ideal.  

\item[$(3)$] Either    
\begin{enumerate}
\item[$(3.1)$] $I$ is a graded prime ideal, in which case $(E\setminus (H,S))^{0}$ is downward directed, or   
\item[$(3.2)$] $I$ is a power of a non-graded prime ideal, in which case $(E\setminus (H,S))^{0}$ is downward directed; and $I = I(H,B_{H}) + \langle p^n(c) \rangle$ for a $($unique$)$ cycle $c$ without exits in $E\setminus (H,B_{H})$, an irreducible polynomial $p(x) \in K[x,x^{-1}]$, and a positive integer $n$.
\end{enumerate}
\end{enumerate}
\end{proposition}

We conclude this section with a couple of facts about ideals in general rings.

\begin{proposition}\label{Morita equivalence} \cite[Proposition 1]{EEKRR}
Let $R$ be a ring with local units and $Y$ a nonempty set.
\begin{enumerate}
\item[$(1)$] Every ideal of  $\, \M_{Y}(R)$ is of the form $\, \M_{Y}(I)$ for some ideal $I$ of $R$. The map $I \mapsto \M_{Y}(I)$ defines a lattice isomorphism between the lattice of ideals of $R$ and the lattice of ideals of $\, \M_{Y}(R)$.

\item[$(2)$] For any two ideals $I$ and $J$ of $R$, we have $\, \M_{Y}(IJ)=\M_{Y} (I)\M_{Y}(J)$.
\end{enumerate}
\end{proposition}

\section{Products of Prime-Power Ideals} \label{prime-prod-section}

In this section we establish the uniqueness of factoring an arbitrary ideal in a Leavitt path algebra into an irredundant product of finitely many prime-power ideals. We begin with a couple of technical lemmas.

\begin{lemma} \label{PrimeUniqueLemma}
Let $L=L_K(E)$ be a Leavitt path algebra, and let $I = P_{1}^{r_{1}}\cdots P_{m}^{r_{m}}$, where each $P_i$ is a prime ideal, $m, r_1, \dots, r_m$ are positive integers, $P_i \not\subseteq P_j$ whenever $i \neq j$, and $r_i = 1$ whenever $P_i$ is graded. Then the $P_i$, the $r_i$, and $m$ are uniquely determined by $I$, up to the order in which the ideals appear in the product.
\end{lemma}

\begin{proof}
Suppose that $I = Q_{1}^{s_{1}} \cdots Q_{n}^{s_{n}}$, where each $Q_i$ is a prime ideal, $n,s_1, \dots, s_n$ are positive integers, $Q_i \not\subseteq Q_j$ whenever $i \neq j$, and $s_i = 1$ whenever $Q_i$ is graded. We shall prove that $m=n$, and that there is a permutation $\sigma$ of $\{1, \dots, m\}$, such that $P_i = Q_{\sigma(i)}$ and $r_i = s_{\sigma(i)}$ for each $i \in \{1, \dots, m\}$.

Let $i \in \{1, \dots, m\}$. Since $P_i$ is prime, $I \subseteq P_i$ implies that $Q_{\sigma(i)} \subseteq P_{i}$ for some $\sigma(i) \in \{1, \dots, n\}$. Likewise, since $I \subseteq Q_{\sigma(i)}$, we have $P_{i'} \subseteq Q_{\sigma(i)}$ for some $i' \in \{1, \dots, m\}$, and hence $P_{i'} \subseteq P_i$. By hypothesis, this means that $i' = i$, and hence $P_i = Q_{\sigma(i)}$. Since $i$ was arbitrary, and since $P_1, \dots, P_m$ are distinct, this implies that $m \leq n$. Then, by symmetry, we conclude that $m=n$, and that $i \mapsto \sigma(i)$ defines a permutation $\sigma$ of $\{1, \dots, m\}$. Setting $t_i = s_{\sigma(i)}$ for each $i$, and using Theorem~\ref{arbitrary ideals}(4), we then have
\[
I=P_{1}^{r_{1}}P_{2}^{r_{2}}\cdots P_{m}^{r_{m}}=P_{1}^{t_{1}}P_{2}^{t_{2}}\cdots P_{m}^{t_{m}}.
\]
So it remains to show that $r_i = t_i$ for each $i$.

By hypothesis, if $P_i$ is graded, for some $i$, then $r_i=1=t_i$. We may therefore assume that at least one of the ideals $P_1, \dots, P_m$ is non-graded, which may be taken to be $P_1$, by Theorem~\ref{arbitrary ideals}(4). (Note that this implies that none of $P_1, \dots, P_m$ is zero.) We shall conclude the proof by showing that $r_1 = t_1$.

By Theorem~\ref{specificprime}, $P_1 = I(H,B_{H})+\langle f_1(c)\rangle$, where $H=P_{1}\cap E^{0}$, $c$ is a cycle without exits in $E\backslash(H,B_{H})$, $E^0\setminus H = M(s(c))$, and $f_1(x) \in K[x, x^{-1}]$ is irreducible (and may be taken to be in $K[x]$, with a nonzero constant term). Note that $(E\backslash(H,B_{H}))^0 = E^0\setminus H$ is downward directed. We shall now pass to the ring $\bar{L}=L/I(H,B_{H})$, which may be identified with $L_{K}(E\backslash(H,B_{H}))$, as mentioned in Section~\ref{LPAidealSect}. For each $i \in \{1, \dots, m\}$ let $\bar{P}_{i}=(P_{i}+I(H,B_{H}))/I(H,B_{H}) \subseteq \bar{L}$, and let $M$ be the ideal of $\bar{L}$ generated by $\{c^{0}\}$. Then, by Theorem~\ref{arbitrary ideals}, $\bar{P}_{1}= \langle f_1(c) \rangle \subsetneq M$. Moreover, by Lemma~\ref{cycle-ideal-lemma}(2,3), for each $i \geq 2$, either $M\subseteq\bar{P}_i$ or $\bar{P}_i\subseteq M$, depending on whether or not $\bar{P}_i$ contains a vertex.

Without loss of generality, we may assume that $\bar{P}_{i}\subsetneq M$ if $1 < i \leq k$, and $M \subseteq \bar{P}_{i}$ if $k < i \leq m$, for some $k$. Then, by Lemma~\ref{Product=Intersection}(1), $M$ being graded implies that $\bar{P}_{i}M=\bar{P}_{i}\cap M=\bar{P}_{i}$ for each $1 < i \leq k$, and $M\bar{P}_i = M\cap\bar{P}_i=M$ for each $k < i \leq m$. Using these equations repeatedly, we then have
\[
\bar{P}_{1}^{r_{1}}\bar{P}_{2}^{r_{2}}\cdots\bar{P}_{m}^{r_{m}} = \bar{P}_{1}^{r_{1}}\bar{P}_{2}^{r_{2}}\cdots \bar{P}_{k}^{r_{k}} M \bar{P}_{k+1}^{r_{k+1}} \cdots \bar{P}_{m}^{r_{m}} = \bar{P}_{1}^{r_{1}}\bar{P}_{2}^{r_{2}}\cdots \bar{P}_{k}^{r_{k}} M=\bar{P}_{1}^{r_{1}}\bar{P}_{2}^{r_{2}}\cdots \bar{P}_{k}^{r_{k}},
\]
and likewise 
\[
\bar{P}_{1}^{t_{1}}\bar{P}_{2}^{t_{2}}\cdots\bar{P}_{m}^{t_{m}} = \bar{P}_{1}^{t_{1}}\bar{P}_{2}^{t_{2}}\cdots\bar{P}_{k}^{t_{k}}.
\] 
Now, for each $1 < i \leq k$, since $\bar{P}_{i}$ contains no vertices, we have $\bar{P}_{i}=\langle g_{i}(c)\rangle$ for some $g_{i}(x)\in K[x]$ with a nonzero constant term, by Lemma~\ref{cycle-ideal-lemma}(2). This means that, for each $1 < i \leq k$, we have $\mathrm{gr}(P_{i}) \subseteq I(H,B_{H})$, which in turn implies that $P_{i}$ is not graded. (Otherwise, $P_{i}=\mathrm{gr}(P_{i})\subseteq I(H,B_{H})\subseteq P_{1}$, contrary to hypothesis.) Thus, by Theorem~\ref{specificprime}, for each $1 < i \leq k$, we have $P_{i}=I(H_{i},B_{H_{i}})+ \langle f_{i}(c_{i}) \rangle$ where $H_i=P_{i}\cap E^{0}$, $c_i$ is a cycle without exits in $E\backslash(H_i,B_{H_i})$, $E^0\setminus H_i = M(s(c_i))$, and $f_i(x) \in K[x, x^{-1}]$ is irreducible (and may be taken to be in $K[x]$, with a nonzero constant term). In particular, $H_i \subseteq H$.

We claim that $H_{i}=H$ for each $1 < i \leq k$. Seeking a contradiction, suppose that there is a vertex $u\in H\backslash H_{i}$ for some $i$. Since $E^0\setminus H_i = M(s(c_i))$, we have $u\geq s(c_i)$. Since $H$ is hereditary, this implies that $s(c_i) \in H$, and hence $\{c_{i}^{0}\} \subseteq H$. It follows that $P_i \subseteq P_1$, contrary to hypothesis. Thus $H_{i}=H$, and hence also $E\backslash(H,B_{H}) = E\backslash(H_{i},B_{H_{i}})$, for each $1 < i \leq k$.

Since $E^0\setminus H = M(s(c))$, and the cycles $c_i$ have no exits in $E\backslash(H,B_{H})$, necessarily $c_i = c$ for each $1 < i \leq k$. Thus $P_{i}=I(H,B_{H})+\langle f_{i}(c) \rangle$ for each $1 < i \leq k$, and the polynomials $f_i(x)$ are necessarily non-conjugate (i.e., are not scalar multiples of one another). By Lemma~\ref{Product}(1), we therefore have
\[
\langle f_1^{r_{1}}(c)\rangle \cdots \langle f_k^{r_{k}}(c)\rangle = \bar{P}_{1}^{r_{1}}\bar{P}_{2}^{r_{2}}\cdots \bar{P}_{k}^{r_{k}} = \bar{I} = \bar{P}_{1}^{t_{1}}\bar{P}_{2}^{t_{2}}\cdots\bar{P}_{k}^{t_{k}} = \langle f_1^{t_{1}}(c)\rangle \cdots \langle f_k^{t_{k}}(c)\rangle.
\]
Setting $v=s(c)$, and noting that $vf_i(c)v=f_i(c)$ for each $i$, we see that $\bar{I} \subseteq v\bar{L}v$. Now, according to~\cite[Lemma 2.2.7]{AAS}, given that $c$ is a cycle without exits in $E\backslash(H,B_{H})$, we have $vL_{K}(E\backslash(H,B_{H}))v \cong K[x,x^{-1}]$, via the $K$-algebra isomorphism that maps $v \mapsto 1$, $c\mapsto x$, and $c^*\mapsto x^{-1}$. In particular, $f_i(c)$ is mapped to $f_i(x)$ by this isomorphism, and so
\[
\langle f_1^{r_{1}}(x) \cdots f_k^{r_{k}}(x)\rangle = \langle f_1^{r_{1}}(x)\rangle \cdots \langle f_k^{r_{k}}(x)\rangle = \langle f_1^{t_{1}}(x)\rangle \cdots \langle f_k^{t_{k}}(x)\rangle = \langle f_1^{t_{1}}(x) \cdots f_k^{t_{k}}(x)\rangle
\]
in $K[x,x^{-1}]$. This implies that $f_1^{r_{1}}(x)$ divides $f_1^{t_{1}}(x) \cdots f_k^{t_{k}}(x)$, and that $f_1^{t_{1}}(x)$ divides $f_1^{r_{1}}(x) \cdots f_k^{r_{k}}(x)$. Since $K[x,x^{-1}]$ is a unique factorization domain, and the $f_i(x)$ are non-conjugate and irreducible, we then conclude that $r_{1}=t_{1}$. By symmetry, it follows that $r_{i}=t_{i}$ for all $i$.
\end{proof}

Recall that a (finite) product $I=J_{1}\cdots J_{m}$ of ideals $J_{i}$ in a ring is \textit{irredundant} if either $m=1$ or $I\neq J_{1}\cdots J_{i-1}J_{i+1} \cdots J_{m}$ for all $i \in \{1, \dots, m\}$. 

We are now ready for our first main result. 

\begin{theorem} \label{Uniqueness} 
Let $L$ be a Leavitt path algebra, and let $I_1, \dots, I_m$ be powers of distinct prime ideals of $L$. If $I_1 \cdots I_m$ is irredundant, then the prime-power ideals $I_1, \dots, I_m$ are uniquely determined, up to the order in which they appear.
\end{theorem}

\begin{proof}
For each $i$, write $I_i = P_{i}^{r_{i}}$ for some prime ideal $P_i$ and positive integer $r_i$, where we can let $r_i = 1$ if $P_i$ is graded, by Lemma~\ref{Product=Intersection}(1). By Lemma~\ref{PrimeUniqueLemma}, to conclude that the $I_i$ are uniquely determined, up to the order in which they appear, it suffices to show that $P_i \not\subseteq P_j$ whenever $i \neq j$ (assuming that $m>1$).

Seeking a contradiction, suppose that $P_{i} \subseteq P_{j}$ for some $i \neq j$. Since these ideals are distinct, by hypothesis, we have $P_i \subsetneq P_j$. Hence, by Lemma~\ref{Properties of primes}, $P_{j}P_{i} = P_{i}$, and so $P_{j}^{r_{j}}P_{i}^{r_{i}} = P_{i}^{r_{i}}$. By Theorem~\ref{arbitrary ideals}(4), we then have
\[
I_{1} \cdots I_{m} = I_{1} \cdots I_{j-1}I_{j+1} \cdots I_{m},
\]
contrary to hypothesis. Therefore $P_i \not\subseteq P_j$ whenever $i \neq j$, as desired.
\end{proof}

\begin{remark}
The assumption that the $I_i$ are powers of distinct prime ideals is necessary for the conclusion of Theorem~\ref{Uniqueness} to hold. For example, if $P$ is a non-graded prime ideal in a Leavitt path algebra $L$, then $P^2 \neq P$, by Lemma~\ref{PrimeUniqueLemma}. Thus $P^2$ and $P\cdot P$ are two distinct irredundant representations of the same ideal as a product of prime-power ideals.
\end{remark}

The following consequence of Theorem~\ref{Uniqueness} significantly strengthens~\cite[Theorem 3.16]{ARRS} and~\cite[Theorem 3.3]{EMR}.

\begin{corollary} \label{PrimeUniqueness} 
Let $L$ be a Leavitt path algebra, and let $I$ be an irredundant product of finitely many prime ideals of $L$. Then the prime ideals in the product $I$ are uniquely determined, up to the order in which they appear.
\end{corollary}

\begin{proof}
By Theorem~\ref{arbitrary ideals}(4) and the hypothesis, we can write $I = J_{1}\cdots J_{m}$, where $J_1, \dots, J_m$ are powers of distinct prime ideals, and the product is irredundant. The desired conclusion now follows from Theorem~\ref{Uniqueness}.
\end{proof}

\section{Intersections of Prime-Power Ideals} \label{prime-intersect-section}

Our next goal is to show that for powers of distinct prime ideals in a Leavitt path algebra, the product coincides with the intersection. This result will be useful for translating statements about products of ideals into statements about intersections. We require two technical lemmas.

\begin{lemma} \label{gcd-lemma}
Let $L=L_K(E)$ be a Leavitt path algebra, let $(H,S)$ be an admissible pair, let $c$ be a cycle without exits in $E \backslash (H,S)$, and let $m$ be a positive integer. For each $i\in \{1, \dots, m\}$ let $I_i = I(H,S) + \langle f_i(c) \rangle$, where $f_i(x) \in K[x,x^{-1}]$ is nonzero. If $\, \gcd(f_{i}(x),f_{j}(x))=1$ for all $i\neq j$, then 
\[
I_1\cdots I_m = I(H,S) + \langle f_1(c) \cdots f_m(c) \rangle = I_1 \cap \cdots \cap I_m.
\]
\end{lemma}

\begin{proof}
For each $i$ write $\bar{I}_i = I_i/I(H,S)$. Then, in $L_{K}(E\backslash(H,S))\cong L_{K}(E)/I(H,S)$, we have $\bar{I}_i  = \langle f_i(c) \rangle \subseteq \langle \{c^{0}\} \rangle$. Thus, in this ring, using Lemma~\ref{Product}(1), we have
\[
\bar{I}_1 \cdots \bar{I}_m = \langle f_1(c) \rangle \cdots \langle f_m(c) \rangle = \langle f_1(c) \cdots f_m(c) \rangle.
\]
Now, the hypothesis that $\gcd(f_{i}(x),f_{j}(x))=1$ for all $i\neq j$ implies that in $K[x,x^{-1}]$ we have $\operatorname{lcm}(f_{1}(x), \dots, f_{m}(x)) = f_{1}(x) \cdots f_{m}(x)$, and so
\[ 
\langle f_{1}(x) \cdots f_{m}(x)\rangle = \langle\operatorname{lcm}(f_{1}(x), \dots, f_{m}(x))\rangle = \langle f_{1}(x) \rangle \cap\cdots\cap\langle f_{m}(x)\rangle.
\]
By Lemma~\ref{cycle-ideal-lemma}(1), Remark~\ref{cycle-ideal-remark}, and Proposition~\ref{Morita equivalence}, we then have
\[
\bar{I}_1 \cdots \bar{I}_m = \langle f_1(c) \cdots f_m(c) \rangle = \langle f_1(c) \rangle \cap \cdots \cap \langle f_m(c) \rangle = \bar{I}_1 \cap \cdots \cap \bar{I}_m.
\]
It follows that $I_1\cdots I_m = I_1 \cap \cdots \cap I_m$ and $I_1\cdots I_m = I(H,S) + \langle f_1(c) \cdots f_m(c) \rangle$, as desired.
\end{proof}

\begin{lemma} \label{distinct-gr-lemma}
Let $L=L_K(E)$ be a Leavitt path algebra, let $m$ be a positive integer, and let $H_1, \dots, H_m$ be distinct hereditary saturated subsets of $E^0$. Also, for each $i \in \{1, \dots, m\}$, let $c_i$ be a cycle without exits in $E\backslash(H_{i},B_{H_i})$ such that $E^{0} \setminus H_i = M(s(c_i))$, and let $I_i = I(H_i,B_{H_i}) + \langle f_i(c_i) \rangle$ for some $f_i(x)\in K[x,x^{-1}]$. Then $I_1\cdots I_m = I_1 \cap \cdots \cap I_m$.
\end{lemma}

\begin{proof}
Since the $c_i$ are without exits in $E\backslash(H_{i},B_{H_i})$, the $H_i$ are distinct, and $E^{0} \setminus H_i = M(s(c_i))$ for each $i$, necessarily the cycles $c_i$ are distinct. Letting $N = I_1 \cap \cdots \cap I_m$, by Theorem~\ref{arbitrary ideals}(1), $N=I(H,S)+\sum_{i\in Y} \langle g_{i}(d_{i})\rangle$, where each $d_i$ is a cycle without exits in $E \setminus (H,S)$, and each $g_i(x) \in K[x]$ has a nonzero constant term. By Lemma~\ref{Product=Intersection}(2), 
\[
I(H,S) = I(H_1,B_{H_1}) \cap \cdots \cap I(H_m,B_{H_m}) = I(H_1,B_{H_1}) \cdots I(H_m,B_{H_m})  \subseteq I_1\cdots I_m.
\]
Since $I_1\cdots I_m \subseteq I_1 \cap \cdots \cap I_m$ for any ideals $I_i$ in any ring, it therefore suffices to show that $\langle g_{i}(d_{i})\rangle \subseteq I_1\cdots I_m$ for arbitrary $i \in Y$. 

It cannot be the case that $d_{i} \in I(H_1,B_{H_1}) \cap \cdots \cap I(H_m,B_{H_m})$, since then we would have $d_{i} \in I(H,S)$, contrary to the choice of $d_i$. Thus, without loss of generality we may assume that $d_{i} \notin I(H_1,B_{H_1})$. Then $d_i$ is a cycle in $E \backslash (H_{1},B_{H_1})$. Since $E\backslash(H_{1},B_{H_1})$ is a subgraph of $E \backslash (H,S)$, we conclude that $d_{i}$ has no exits in $E \backslash (H_{1},B_{H_1})$. The hypotheses on $c_1$ then imply that $d_{i}=c_{1}$. Thus it cannot be the case that $d_{i} \notin I(H_j,B_{H_j})$ for some $j > 1$, since otherwise the same argument would show that $d_i=c_j$, contradicting the distinctness of the cycles $c_j$. Therefore, by Lemma~\ref{Product=Intersection}(1), we have
\[
g_{i}(d_{i}) \in I_1 \cap \bigcap_{j = 2}^m I(H_j,B_{H_j}) = I_1 \cdot \prod_{j = 2}^m I(H_j,B_{H_j})\subseteq I_1\cdots I_m,
\] 
which implies that $\langle g_{i}(d_{i})\rangle \subseteq I_1\cdots I_m$, as desired.
\end{proof}

We are now ready to show that for powers of distinct prime ideals in a Leavitt path algebra, the product coincides with the intersection. A version of this result, for row-finite graphs, can be found in an unpublished thesis of van den Hove~\cite[Corollary 5.3.3]{H}, where the proof uses lattice-theoretic properties of ideals, represented as ordered triples of vertices, cycles, and polynomials.

\begin{proposition} \label{prime-power-intersect} 
Let $L=L_K(E)$ be a Leavitt path algebra, let $m, r_{1}, \dots, r_{m}$ be positive integers, and let $P_{1}, \dots, P_{m}$ be distinct prime ideals of $L$. Then 
\[
P_{1}^{r_{1}}\cdots P_{m}^{r_{m}} = P_{1}^{r_{1}}\cap\cdots\cap P_{m}^{r_{m}}.
\]
\end{proposition}

\begin{proof}
Upon reindexing, we may assume that $P_{k+1}, \dots, P_m$ are graded, and $P_1, \dots, P_k$ are not, for some $0 \leq k \leq m$. Then, by Lemma~\ref{Product=Intersection}(1), 
\[
(P_{1}^{r_{1}}\cap \cdots \cap P_{k}^{r_{k}})P_{k+1}^{r_{1}} \cdots P_{m}^{r_{m}} = P_{1}^{r_{1}}\cap\cdots\cap P_{m}^{r_{m}},
\]
and so, by Theorem~\ref{arbitrary ideals}(4), it suffices to show that 
\[
P_{1}^{r_{1}}\cdots P_{k}^{r_{k}} = P_{1}^{r_{1}}\cap\cdots\cap P_{k}^{r_{k}},
\]
where we may assume that $k\geq 2$.

Now, suppose, upon further reindexing, that $\mathrm{gr}(P_1) = \dots = \mathrm{gr}(P_n) = I(H,B_H)$ for some $1 < n \leq k$. By Theorem~\ref{specificprime}, there can be only one cycle $c$ with no exits in $E\backslash(H,B_H)$, and so $P_{i} = I(H,B_H) + \langle p_i(c) \rangle$, where $p_i(x) \in K[x,x^{-1}]$ is irreducible, for each $1 \leq i \leq n$. Since the $P_i$ are distinct, necessarily $\gcd(p_{i}(x),p_{j}(x))=1$, and hence also $\gcd(p_{i}(x)^{r_{i}},p_{j}(x)^{r_{j}})=1$, for all $i\neq j$. Therefore, applying Lemma~\ref{gcd-lemma} to the $P_{i}^{r_{i}} = I(H,B_H) + \langle p_{i}(x)^{r_{i}} \rangle$, we conclude that 
\[
P_{1}^{r_{1}}\cdots P_{n}^{r_{n}} = P_{1}^{r_{1}}\cap\cdots\cap P_{n}^{r_{n}} = I(H,B_H) + \langle p_{1}(c)^{r_{1}} \cdots p_{n}(c)^{r_{n}} \rangle.
\]
In this fashion, using Theorem~\ref{arbitrary ideals}(4), we can group the ideals $P_{i}^{r_{i}}$ in $P_{1}^{r_{1}}\cdots P_{k}^{r_{k}}$ having the same graded part, to obtain a product of ideals of the from $I(H,B_H) + \langle f(c) \rangle$, which satisfy the hypotheses of Lemma~\ref{distinct-gr-lemma}. The desired conclusion then follows from that lemma.
\end{proof}

\begin{remark}
The previous result implies that an ideal in a Leavitt path algebra is the product of prime-power ideals if and only if it is the intersection of finitely many prime-power ideals. In contrast to this, a product of prime ideals may not be expressible as an intersection of prime ideals. For example, if $P$ is a non-graded prime ideal in a Leavitt path algebra, then $P^2$ is not the intersection of any collection of prime ideals, according to~\cite[Lemma 4.9]{AMR}. (This can also be concluded from Theorem~\ref{uniqueness-intersect}.)
\end{remark}

\begin{remark}
We note also that Proposition~\ref{prime-power-intersect} describes a property that is special to Leavitt path algebras, which does not in general hold even for commutative unital rings. For example, let $K$ be a field, let $R = K[x,y]$ be the polynomial ring in two variables, and consider the ideals $I = \langle x,y \rangle$ and $J = \langle x \rangle$ of $R$. Then it is easy to see that $I$ and $J$ are prime (and distinct), but $ IJ = \langle x^2,xy \rangle  \neq J = I \cap J$.
\end{remark}

Recall that given a nonempty collection $\{S_i\mid i \in Y\}$ of sets, the intersection $\bigcap_{i\in Y} S_i$ is \emph{irredundant} if either $|Y|=1$, or $\bigcap_{i\in Y \setminus \{j\}} S_i \not\subseteq S_j$ for all $j \in Y$. Similarly, the union $\bigcup_{i\in Y} S_i$ is \emph{irredundant} if either $|Y|=1$, or $S_j \not\subseteq \bigcup_{i\in Y \setminus \{j\}} S_i$ for all $j \in Y$.

We are now ready for the main result of this section, which is an analogue of Theorem~\ref{Uniqueness} for intersections.

\begin{theorem} \label{uniqueness-intersect}
Let $L$ be a Leavitt path algebra, and let $I$ be an irredundant intersection of finitely many prime-power ideals of $L$. Then the ideals in the intersection $I$ are uniquely determined.
\end{theorem}

\begin{proof}
Suppose that
\[
P_{1}^{r_{1}}\cap\cdots\cap P_{m}^{r_{m}} = Q_{1}^{s_{1}} \cap\cdots\cap Q_{n}^{s_{n}}
\]
are irredundant intersections, where $m, n, r_i, s_i$ are positive integers, and the $P_i$ and $Q_i$ are prime ideals of $L$. By Lemma~\ref{Product=Intersection}(1), $P_i^{r_i} = P_i$ whenever $P_i$ is graded, and likewise for the $Q_i$. So we may assume that $r_i=1$ whenever $P_i$ is graded, and likewise for the $Q_i$. We may also assume that $m>1$, since otherwise there is nothing to prove.

Next, suppose that $P_{i} \subseteq P_{j}$ for some $i \neq j$. Since the intersection is irredundant, the $P_i$ must be distinct, and so $P_{i} \subsetneq P_{j}$. Then, by Lemma~\ref{Properties of primes}, $P_{j}P_{i} = P_{i}$. It follows that $P_{j}^{r_{j}}P_{i}^{r_{i}} = P_{i}^{r_{i}}$, which in turn implies that $P_{j}^{r_{j}}\cap P_{i}^{r_{i}} = P_{i}^{r_{i}}$. We can then replace $P_{j}^{r_{j}}\cap P_{i}^{r_{i}}$ with $P_{i}^{r_{i}}$ in the intersection $P_{1}^{r_{1}}\cap \cdots\cap P_{m}^{r_{m}}$, which contradicts the supposition that it is irredundant. Thus $P_i \not\subseteq P_j$ for all $i \neq j$, and likewise for the $Q_i$.

Since the $P_i$ are distinct, as are the $Q_i$, by Proposition~\ref{prime-power-intersect}, we have
\[
P_{1}^{r_{1}}\cdots P_{m}^{r_{m}} = Q_{1}^{s_{1}} \cdots Q_{n}^{s_{n}}.
\]
Then Lemma~\ref{PrimeUniqueLemma} implies that $m=n$, and that there is a permutation $\sigma$ of $\{1, \dots, m\}$, such that $P_i = Q_{\sigma(i)}$ and $r_i = s_{\sigma(i)}$ for each $i \in \{1, \dots, m\}$, giving the desired conclusion.
\end{proof}

We conclude this section with a couple of technical lemmas, which will useful subsequently. The first is an expanded version of part of the proof of~\cite[Theorem 6.2]{R-2}.

\begin{lemma} \label{unique-cycle-lemma}
Let $L=L_K(E)$ be a Leavitt path algebra, let $I$ be a non-graded ideal of $L$, and write $I=I(H,S)+\sum_{i\in Y} \langle f_{i}(c_{i})\rangle$, using the notation of Theorem \ref{arbitrary ideals}(1). Suppose further that $I(H,S)= \bigcap_{i=1}^{m} P_{i}$ for some graded prime ideals $P_{i}=I(H_{i},S_{i})$. Then the following hold.
\begin{enumerate}
\item[$(1)$] For each $i \in Y$ there is an $l \in \{1, \dots, m\}$ such that $s(c_{i}) \notin P_{l}$. Moreover, $s(c_{j}) \in P_{l}$ for all $j \neq i$, $S_{l}= B_{H_{l}}$, and $u\geq s(c_{i})$ for every $u \in E^{0}\backslash H_{l}$.

\item[$(2)$] $|Y| \leq m$. 

\item[$(3)$] Given $i \in Y$, if $l, k \in \{1, \dots, m\}$ are such that $s(c_{i}) \notin P_{l} \cup P_{k}$, then $P_{l}\cap P_{k}$ is a graded prime ideal.
\end{enumerate}

Moreover, if $\, \bigcap_{i=1}^{m} P_{i}$ is an irredundant intersection, then the following hold.
\begin{enumerate}
\item[$(4)$] Upon reindexing the $c_{i}$ and $P_{i}$, and letting $|Y| = n \leq m$, for each $i \in \{1, \dots, n\}$ we have $s(c_{i}) \notin P_{i}$ and $s(c_{i}) \in P_{j}$ for all $j \in \{1, \dots, m\} \setminus \{i\}$.

\item[$(5)$] $I = \prod_{i=1}^n (I(H_{i},B_{H_{i}}) + \langle f_{i}(c_{i})\rangle) \cdot \prod_{i=n+1}^m I(H_{i},B_{H_{i}})$, where the $P_{i} = I(H_{i},B_{H_{i}})$ are indexed as in $(4)$.
\end{enumerate}
\end{lemma}

\begin{proof}
(1) Let $i \in Y$. Then there is an $l \in \{1, \dots, m\}$ such that $s(c_{i}) \notin P_{l}$, since otherwise we would have $s(c_{i}) \in \bigcap_{i=1}^{m} P_{i}$, and hence $c_{i} \in I(H,S)$, contrary to hypothesis.

Since $c_{i}$ is a cycle without exists in $E\setminus (H,S)$, and $(H,S) \leq (H_{l},S_{l})$, it follows that $c_{i}$ has no exits in $E\setminus (H_{l},S_{l})$ as well. Now, by Theorem~\ref{specificprime}, $E^{0}\backslash H_{l}$ is downward directed, and so we must have $u\geq s(c_{i})$ for every $u \in E^{0}\backslash H_{l}$. Then it cannot be the case that $s(c_{i}) \in B_{H_{l}}$ and $S_{l} = B_{H_{l}} \setminus \{s(c_{i})\}$, since in that situation $c_{i}$ would have an exit in $E\setminus (H_{l},S_{l})$, by construction. Thus, by Theorem \ref{specificprime}, $S_{l} = B_{H_{l}}$. Moreover, since $E^{0}\backslash H_{l}$ is downward directed, it cannot be the case that $s(c_{j}) \in E^{0}\backslash H_{l}$ for some cycle $c_{j}$ (without exists) that is different from $c_{i}$. Hence $s(c_{j}) \in H_{l} \subseteq P_{l}$ for all $j \neq i$. 

(2) This follows immediately from (1).

(3) Suppose that $s(c_{i}) \notin P_{l} \cup P_{k}$ for some $l, k\in \{1, \dots, m\}$, and let $P'=P_{l}\cap P_{k}$. Then, by (1), $u\geq s(c_{i})$ for every $u \in (E^{0}\backslash H_{l}) \cup (E^{0}\backslash H_{k})$, $S_{l} = B_{H_{l}}$, and $S_{k} = B_{H_{k}}$. Being an intersection of graded ideals, $P'$ is graded, and so $P' = I(H', S')$ for some admissible pair $(H', S')$. Thus, by Theorem~\ref{arbitrary ideals}(3), we have 
\[
H' = P' \cap E^{0} = (P_{l}\cap E^{0}) \cap (P_{k}\cap E^{0}) = H_{l}\cap H_{k},
\] 
and hence 
\[
(E^{0}\backslash H_{l}) \cup (E^{0}\backslash H_{k}) = E^{0}\backslash (H_{l}\cap H_{k}) = E^{0}\backslash H'.
\] 
Since $u\geq s(c_{i})$ for every $u \in (E^{0}\backslash H_{l}) \cup (E^{0}\backslash H_{k}) = E^{0} \backslash H'$, it follows that $E^{0} \backslash H'$ is downward directed. Thus to conclude that $P'$ is a (graded) prime ideal, it suffices to show that $S' = B_{H'}$, by Theorem~\ref{specificprime}. 

Consider $u \in B_{H'}$. Since $u \in E^{0}\backslash H'$, we have $u\geq s(c_{i})$, from which it follows that $u \in B_{H_l} \cap B_{H_k}$. Let $e_1, \dots, e_{n_{l}}, f_1, \dots, f_{n_{k}}, g_1, \dots, g_{n'} \in E^1$ be all the edges having source $u$ but range in $E^{0} \setminus H'$, where each $r(e_i) \in H_{l} \setminus H_{k}$, $r(f_i) \in H_{k} \setminus H_{l}$, and $r(g_i) \notin H_{l} \cup H_{k}$. Then for each $i$ we have $e_ie^*_i \in P_{l}$ and $f_if^*_i \in P_{k}$, and so 
\[
u^{H'} = u - \sum_{i=1}^{n_{l}}e_ie^*_i - \sum_{i=1}^{n_{k}}f_if^*_i - \sum_{i=1}^{n'}g_ig^*_i = u^{H_{l}} - \sum_{i=1}^{n_{l}}e_ie^*_i = u^{H_{k}} - \sum_{i=1}^{n_{k}}f_if^*_i \in P_{l}\cap P_{k} = P'.
\]
Therefore, $u^{H'} \in P'$ for all $u \in B_{H'}$, and so $S' = B_{H'}$, as desired.

For the remainder of the proof we shall assume that $\bigcap_{i=1}^{m} P_{i}$ is an irredundant intersection.

(4) Suppose that $s(c_{i}) \notin P_{l}$ and $s(c_{i}) \notin P_{k}$, for some $i \in Y$ and distinct $l,k \in \{1, \dots, m\}$. Then, by (3), $P'=P_{l}\cap P_{k}$ is a graded prime ideal. Hence, by Lemma~\ref{Product=Intersection}(2), $P' = P_{l}P_{k}$, and so either $P_{l} \subseteq P'$ or $P_{k} \subseteq P'$, which means that either $P_{l} \subseteq P_{k}$ or $P_{k} \subseteq P_{l}$. This contradicts the hypothesis that $\bigcap_{i=1}^{m} P_{i}$ is irredundant, and hence $s(c_{i}) \notin P_{j}$ for at most one $j\in \{1, \dots, m\}$, from which (4) follows, by (1).

(5) For each $i \in \{1, \dots, n\}$ let $Q_{i} = I(H_{i},B_{H_{i}}) + \langle f_{i}(c_{i})\rangle$, and let $J = \bigcap_{j=n+1}^{m} P_{j}$. (If $n=m$, then we take $J=L$.) Then, by (4), we have $s(c_i) \in J$ for all $i \in \{1, \dots, n\}$, and hence $I \subseteq J$. We now proceed by induction on $n$. 

If $n=1$, then by Lemma~\ref{Product=Intersection}(1) and the fact that ideal multiplication distributes over ideal addition in any ring,
\[
Q_{1}P_{2}\cdots P_{m} = Q_{1} J = (P_{1} + \langle f_{1}(c_{1})\rangle ) J = P_{1}J + \langle f_{1}(c_{1})\rangle J = \bigcap_{i=1}^{m} P_{i} + \langle f_{1}(c_{1})\rangle J.
\]
Since $I \subseteq J$, by Lemma~\ref{Product=Intersection}(1), we see that $\langle f_{1}(c_{1})\rangle J = \langle f_{1}(c_{1})\rangle$, and hence
\[
Q_{1}P_{2}\cdots P_{m} = I(H,S)+ \langle f_{1}(c_{1})\rangle = I.
\]

So let us suppose that $n>1$, and the statement holds for $n-1$. That is, 
\[
Q_{1}\cdots Q_{n-1}P_{n+1}\cdots P_{m} = J \cap \bigcap_{i=1}^{n-1} P_{i} + \sum_{i=1}^{n-1} \langle f_{i}(c_{i})\rangle.
\]
Then, using Theorem~\ref{arbitrary ideals}(4) and Lemma~\ref{Product=Intersection}(1), we have
\begin{align*}
Q_{1}\cdots Q_{n}P_{n+1}\cdots P_{m} & = Q_{1}\cdots Q_{n-1}(P_{n} + \langle f_{n}(c_{n})\rangle ) P_{n+1}\cdots P_{m}\\
& = Q_{1}\cdots Q_{n-1}P_{n}P_{n+1}\cdots P_{m} + Q_{1}\cdots Q_{n-1}\langle f_{n}(c_{n})\rangle J\\
&=\bigg(J \cap \bigcap_{i=1}^{n-1} P_{i} + \sum_{i=1}^{n-1} \langle f_{i}(c_{i})\rangle\bigg) P_{n} + Q_{1}\cdots Q_{n-1}\langle f_{n}(c_{n})\rangle\\
& = \bigcap_{i=1}^{m} P_{i} + \sum_{i=1}^{n-1} \langle f_{i}(c_{i})\rangle P_{n} + \bigg(\prod_{i=1}^{n-1} (P_{i} + \langle f_{i}(c_{i})\rangle) \bigg) \langle f_{n}(c_{n})\rangle.
\end{align*}
Since $s(c_{i}) \in P_{j}$ for all $i\neq j$, again using Lemma~\ref{Product=Intersection}(1), we see that 
\begin{align*}
Q_{1}\cdots Q_{n}P_{n+1}\cdots P_{m} & = I(H,S) + \sum_{i=1}^{n-1} \langle f_{i}(c_{i})\rangle + \bigg(\prod_{i=1}^{n-1} P_{i}\bigg)\langle f_{n}(c_{n})\rangle + \bigg(\prod_{i=1}^{n-1}\langle f_{i}(c_{i})\rangle \bigg)\langle f_{n}(c_{n})\rangle \\
& = I(H,S) + \sum_{i=1}^{n} \langle f_{i}(c_{i})\rangle + \bigg( \prod_{i=1}^{n-1}\langle f_{i}(c_{i})\rangle \bigg) \langle f_{n}(c_{n})\rangle
\end{align*}
Finally, $\langle f_{i}(c_{i})\rangle\langle f_{j}(c_{j})\rangle \subseteq I(H,S)$ for $i \neq j$, by Lemma~\ref{Product}(2) and Theorem~\ref{arbitrary ideals}(2). Thus
\[
Q_{1}\cdots Q_{n}P_{n+1}\cdots P_{m} = I(H,S) + \sum_{i=1}^{n} \langle f_{i}(c_{i} \rangle = I,
\]
as desired.
\end{proof}

\begin{lemma} \label{prime-gr-part}
Let $L=L_K(E)$ be a Leavitt path algebra, and let $I$ be a non-graded ideal of $L$ such that $\, \mathrm{gr}(I)$ is a prime ideal. Then there exist non-graded prime ideals $P_{1}, \dots, P_{m}$ and positive integers $m, r_1, \dots, r_m$ such that 
\[
I=P_{1}^{r_{1}}\cdots P_{m}^{r_{m}} = P_{1}^{r_{1}}\cap \cdots \cap P_{m}^{r_{m}},
\] 
and $\, \mathrm{gr}(P_{i}) = \mathrm{gr}(I)$ for each $i$.
\end{lemma}

\begin{proof}
By Lemma~\ref{unique-cycle-lemma}(2,5), we have $I = I(H,B_{H}) + \langle f(c)\rangle$, where $\mathrm{gr}(I) = I(H,B_{H})$ is a prime ideal, $c$ is a cycle without exits in $E\backslash(H,B_{H})$, and $f(x)\in K[x]$ is a polynomial with a nonzero constant term.

Write $f(x)=p_{1}^{r_{1}}(x)\cdots p_{m}^{r_{m}}(x)$ for some non-conjugate irreducible polynomials $p_i(x) \in K[x]$ and positive integers $m, r_1, \dots, r_m$. Then, using Lemma \ref{Product}(1) and Theorem \ref{arbitrary ideals}(2), we have
\[
I=(I(H,B_{H})+\langle p_1(c)\rangle)^{r_{1}} \cdots (I(H,B_{H})+\langle p_m(c)\rangle)^{r_{m}}.
\]
Writing $P_{i}=I(H,B_{H})+ \langle p_{i}(c)\rangle$ for each $i$, we then have $I=P_{1}^{r_{1}}\cdots P_{m}^{r_{m}}$, where each $P_i$ is prime, by Theorem~\ref{specificprime}. By construction, the $P_i$ are distinct, and so $I = P_{1}^{r_{1}}\cap \cdots \cap P_{m}^{r_{m}}$, by Proposition~\ref{prime-power-intersect}.
\end{proof}

\section{Completely Irreducible Ideals} \label{comp-irred-section}

Recall that a proper ideal $I$ of a ring $R$ is \textit{completely irreducible} if $I$ is not the intersection of any set of ideals properly containing $I$. In this section we characterize the completely irreducible ideals of an arbitrary Leavitt path algebra. These ideals turn out to be prime-power ideals of a special sort, and so the results from the previous two sections apply to them.

It is well-known and easy to prove that an ideal $I$ of a ring $R$ is completely irreducible if and only if there exists $r \in R$ such that $I$ is an ideal maximal with respect to $r \notin I$. An straightforward consequence of this statement is the following observation, which will be useful throughout the rest of the paper.

\begin{proposition} \label{Everyone is int comp.irr.} 
Every proper ideal in a ring is the intersection of the completely irreducible ideals that contain it.
\end{proposition}

The following concepts will help us describe the completely irreducible ideals in Leavitt path algebras.

\begin{definition}    
Let $E$ be a graph, and let $S$ be a nonempty subset of $E^0$ .
\begin{enumerate}
\item[$(1)$] We say that $S$ satisfies the \textit{countable separation property} (\textit{CSP} for short) if there is a countable subset $T$ of $S$, such that for every $u\in S$ there is a $v\in T$ satisfying $u\geq v$.

\item[$(2)$] We say that $S$ satisfies the \textit{strong CSP} if $S$ satisfies the CSP with respect to some countable subset $T$, such that $T$ is contained in every nonempty hereditary saturated subset of $S$.
\end{enumerate}
\end{definition}

We next describe the Leavitt path algebras for which the zero ideal is completely irreducible.

\begin{proposition} \label{Zero Comp Irreducible} 
The following are equivalent for any Leavitt path algebra $L=L_K(E)$.
\begin{enumerate}
\item[$(1)$] The zero ideal is completely irreducible.

\item[$(2)$] The graph $E$ satisfies condition (L), and $E^{0}$ is downward directed and satisfies the strong CSP.
\end{enumerate}
\end{proposition}

\begin{proof}
(1) $\Rightarrow$ (2) Suppose that $\{0\}$ is completely irreducible. According to~\cite[Proposition 2.3.2]{AAS}, $\{0\}$ is the Jacobson radical of $L$, and hence $\{0\}$ is also the intersection of all the primitive ideals of $L$. (It is a standard fact that in any ring, the Jacobson radical is the intersection of all the primitive ideals. This is typically proved for unital rings, but can be easily extended to rings with local units.) We therefore conclude that $\{0\}$ must be primitive, and hence $L$ is a primitive ring. By \cite[Theorem 5.7]{ABR}, $L$ being primitive implies that $E$ satisfies condition (L), $E^{0}$ is downward directed, and $E^{0}$ satisfies the CSP, with respect to a countable nonempty subset $S$ of $E^{0}$. Let $\{H_{i} \mid i\in Y\}$ be the collection of all nonempty hereditary saturated subsets of $E^{0}$, and let $H = \bigcap_{i\in Y} H_{i}$. We shall show that $E^{0}$ satisfies the strong CSP with respect to $S\cap H$.

First, we claim that $H \neq \emptyset$. If on the contrary, $H = \emptyset$, then, by Theorem~\ref{arbitrary ideals}(3),
\[
E^{0} \cap \bigcap_{i\in Y} \langle H_{i}\rangle = \bigcap_{i\in Y} H_{i} = H = \emptyset.
\]
As the intersection of graded ideals, $\bigcap_{i\in Y} \langle H_{i}\rangle$ is itself graded, and hence $E^{0} \cap \bigcap_{i\in Y} \langle H_{i}\rangle = \emptyset$ implies that $\bigcap_{i\in Y} \langle H_{i}\rangle = \{0\}$, again, by Theorem \ref{arbitrary ideals}(3). But this contradicts $\{0\}$ being completely irreducible, and hence $H \neq \emptyset$.

To conclude the proof, let us take an arbitrary vertex $v \in E^{0}$ and show that $v\geq w$ for some $w \in S\cap H$. Letting $u\in H$ be any vertex, $E^{0}$ satisfying the CSP with respect to $S$ implies that there exist $v',u'\in S$ such that $v\geq v^{\prime}$ and $u\geq u^{\prime}$. Since $E^{0}$ is downward directed, there exists $w' \in E^{0}$ such that $u'\geq w'$ and $v' \geq w'$. Then, invoking the CSP again, $w'\geq w$ for some $w\in S$. Since $H$ is hereditary, $u\geq w$ implies that $w\in S\cap H$, as desired.

(2) $\Rightarrow$ (1) Suppose that $E$ satisfies condition (L), $E^{0}$ is downward directed, and $E^{0}$ satisfies the strong CSP with respect to some nonempty countable $S \subseteq E^{0}$. According to \cite[Proposition 2.2.14]{AAS}, $E$ satisfying condition (L) implies that every nonzero ideal $I$ of $L$ contains a vertex, and hence $I\cap E^{0} \neq \emptyset$. By \cite[Lemma 2.4.3]{AAS}, the set $I\cap E^{0}$ is hereditary and saturated, and hence, by hypothesis, $S\subseteq I\cap E^{0}$, for every nonzero ideal $I$. Thus the intersection of all the nonzero ideals of $L$ contains $S$, and hence is nonzero. This shows that $\{0\}$ is completely irreducible.
\end{proof}

\begin{example}
Let $E$ be the following graph.
\[
\xymatrix{\ar@{.}[r] & {\bullet}^{v_3} \ar [r]  & {\bullet}^{v_2} \ar [r] & {\bullet}^{v_1}}
\]
Then clearly $E$ satisfies condition (L), $E^{0}$ is downward directed, and $E^{0}$ satisfies the strong CSP with respect to $\{v_1\}$. Thus, the ideal $\{0\}$ is completely irreducible in $L_{K}(E)$, by the previous proposition.
\hfill $\Box$
\end{example}

We now utilize Proposition~\ref{Zero Comp Irreducible} to prove the main result of this section, which describes all the completely irreducible ideals in an arbitrary Leavitt path algebra. (See Proposition~\ref{quasi-primaryequivalent} for a description of the irreducible ideals.)

\begin{theorem} \label{comp.Irreducible} 
Let $L=L_K(E)$ be a Leavitt path algebra, and let $I$ be a proper ideal of $L$. Then $I$ is completely irreducible if and only if exactly one of the following conditions holds.
\begin{enumerate}
\item[$(1)$] $I=I(H,S)$ is a graded ideal, $E\backslash(H,S)$ satisfies condition (L), and $(E\backslash(H,S))^{0}$ is downward directed and satisfies the strong CSP. (In this case $I$ is prime.)

\item[$(2)$] $I=P^{n}$ for some non-graded prime ideal $P$ and positive integer $n$.
\end{enumerate}
\end{theorem}

\begin{proof}
Suppose that $I$ is completely irreducible. Then, in particular, $I$ is irreducible, and so, by Proposition~\ref{quasi-primaryequivalent}, $I=P^{n}$ for some prime ideal $P$ and positive integer $n$. If $P$ is non-graded, then condition (2) holds. Let us therefore suppose that $P$ is graded. Then $I=P^{n}=P$, by Lemma~\ref{Product=Intersection}(1), and hence $P = I = I(H,S)$, for some admissible pair $(H,S)$. Since $I$ is completely irreducible, so is the zero ideal of $L_{K}(E\backslash(H,S))\cong L/I$. Hence condition (1) holds, by Proposition~\ref{Zero Comp Irreducible}. Note that, by Theorem~\ref{arbitrary ideals}(1,2) and Lemma~\ref{Product}, a power of a non-graded ideal is non-graded, and so conditions (1) and (2) are mutually exclusive.

For the converse, first suppose that condition (1) holds. Then, by Proposition \ref{Zero Comp Irreducible}, the zero ideal of $L_{K}(E\backslash (H,S))\cong L/I$ is completely irreducible, and hence so is $I$.

Now let us suppose that condition (2) holds. Then, by Proposition~\ref{quasi-primaryequivalent}, $I = P^n = I(H,B_{H})+\langle p^{n}(c)\rangle$, where $c$ is a cycle without exits in $E\backslash (H,B_{H})$, $p(x) \in K[x,x^{-1}]$ is irreducible, and
$(E\backslash (H,B_{H}))^{0}$ is downward directed. We shall show that $\bar{P}^{n}=P^{n}/I(H,B_{H})$ is completely irreducible in $\bar{L}=L/I(H,B_{H})\cong L_{K}(E\backslash(H,B_{H}))$, from which it follows that $I$ is completely irreducible in $L$.

First, we note that in the principal ideal domain $K[x,x^{-1}]$, if $\langle f(x) \rangle \supseteq \langle p^{n}(x)\rangle$ for some $f(x) \in K[x,x^{-1}]$, then $p^{n}(x)=f(x)g(x)$ for some $g(x)\in K[x,x^{-1}]$, which implies that $\langle f(x)\rangle=\langle p^{k}(x)\rangle$ for some $0 \leq k\leq n$, since $p(x)$ is irreducible. Thus, the set of proper ideals of $K[x,x^{-1}]$ containing $\langle p^{n}(x)\rangle$ is precisely $\{\langle p^{k}(x)\rangle \mid 1 \leq k \leq n\}$. Now let $M$ be the ideal of $\bar{L}$ generated by $\{c^{0}\}$. Then $\bar{P}\subseteq M$, and $M\cong \M_{Y}(K[x,x^{-1}])$ for some index set $Y$, by  Lemma~\ref{cycle-ideal-lemma}(1). So, by Proposition~\ref{Morita equivalence} and Remark \ref{cycle-ideal-remark}, the set of proper ideals of $M$ containing $\bar{P}^{n}=\langle p^{n}(c)\rangle$ is $\{\langle p^{k}(c)\rangle \mid 1 \leq k \leq n\}$. 

Next, let $J$ be an ideal of $\bar{L}$ such that $\bar{P}^{n} \subsetneq J$. Then, by Lemma~\ref{cycle-ideal-lemma}(2,3), either $M \subseteq J$ or $J =\langle h(c)\rangle \subseteq M$ for some $h(x)\in K[x]$. Thus, by the previous paragraph, either $M \subseteq J$ or $J =\langle p^{k}(c)\rangle$ for some $k \in \{1, \dots, n-1\}$. It follows that the intersection of all the ideals of $\bar{L}$ properly containing $\bar{P}^{n}$ is $\langle p^{n-1}(c)\rangle$. This shows that $\bar{P}^{n}$ is completely irreducible in $\bar{L}$, as desired.
\end{proof}

\begin{remark}\label{graded prime comp irred} 
We note that the prime ideals of type (2) in Theorem~\ref{specificprime} are always completely irreducible. More specifically, these prime ideals are of the form $I(H,B_{H}\backslash\{u\})$ for some $u \in B_{H}$, where $E^0 \setminus H = M(u)$. In this situation $(E\backslash (H,B_{H}\backslash\{u\}))^{0} = (E^0\setminus H) \cup \{u'\}$, where $u'$ is a sink. Thus the conditions in (1) of Theorem~\ref{comp.Irreducible} are clearly satisfied.
\end{remark}

\begin{remark}
Theorems~\ref{Uniqueness} and~\ref{comp.Irreducible} imply that  in an irredundant product of completely irreducible ideals, those ideals are uniquely determined, provided that they are powers of distinct prime ideals.

Likewise, Theorems~\ref{uniqueness-intersect} and~\ref{comp.Irreducible} imply that in an irredundant intersection of finitely many completely irreducible ideals, those ideals are uniquely determined. However, generally speaking, representations of an ideal in a ring $R$ as an irredundant intersection of an infinite collection of completely irreducible ideals are not unique--see, e.g., \cite[Example 3.1]{FHO}.
\end{remark}

The rest of this section is devoted to describing Leavitt path algebras where all proper ideals satisfy some condition related to being completely irreducible. We begin with algebras where all proper ideals are completely irreducible.

\begin{theorem}\label{everyidealcompirred} 
The following are equivalent for any Leavitt path algebra $L=L_K(E)$.
\begin{enumerate}
\item[$(1)$] Every proper ideal of $L$ is completely irreducible.

\item[$(2)$] Every ideal of $L$ is graded, and the ideals of $L$ are well-ordered under set inclusion.

\item[$(3)$] The graph $E$ satisfies condition (K), the admissible pairs $(H,S)$ form a chain under the partial order of admissible pairs, and $(E\backslash(H,S))^{0}$ satisfies the strong CSP for each admissible pair $(H,S)$ with $H \neq E^{0}$. 
\end{enumerate}
\end{theorem}  

\begin{proof}   
(1) $\Rightarrow$ (2) Suppose that every proper ideal of $L$ is completely irreducible. Then, by Proposition \ref{quasi-primaryequivalent}, every proper ideal of $L$ is a power of some prime ideal. Hence, according to~\cite[Theorem 4.1]{AMR}, every ideal of $L$ is graded, and the ideals form a chain under set inclusion. Now let $Y$ be a nonempty set of ideals of $L$, and let $J = \bigcap_{I\in Y}I$. If $J \notin Y$, then $J$ must be the intersection of all the ideals properly containing $J$, and hence not completely irreducible, contrary to hypothesis. Therefore $J \in Y$, showing that $Y$ has a least element under set inclusion. It follows that the ideals of $L$ are well-ordered.

(2) $\Rightarrow$ (1) Let us take a proper ideal $J$ of $L$, and show that it is completely irreducible, assuming (2). If $J$ is maximal, then this is trivially the case, and so we may assume that $J$ is not maximal. Let $Y$ be the set of all ideals of $L$ that properly contain $J$. Then, by hypothesis, $\bigcap_{I \in Y} I \in Y$, and hence $J \subsetneq \bigcap_{I \in Y} I$. It follows that $J$ is completely irreducible.

(1) $\Rightarrow$ (3) Suppose that (1) holds. Then, by Proposition~\ref{quasi-primaryequivalent}, every proper ideal of $L$ is a power of some prime ideal. Hence, by \cite[Theorem 4.1]{AMR}, every ideal of $L$ is graded, $E$ satisfies condition (K), and the admissible pairs $(H,S)$ form a chain under the partial order of admissible pairs. Given that every proper ideal of $L$ is completely irreducible and graded, applying Theorem \ref{comp.Irreducible}, we conclude that $(E\backslash(H,S))^{0}$ satisfies the strong CSP for each admissible pair $(H,S)$ with $H \neq E^{0}$.

(3) $\Rightarrow$ (1) Assuming that (3) holds, again, by~\cite[Theorem 4.1]{AMR}, every proper ideal of $L$ is graded and prime. Let $I$ be a proper ideal of $L$, and write $I = I(H,S)$, where $(H,S)$ is an admissible pair. Then, by Proposition~\ref{quasi-primaryequivalent}, $(E\backslash (H,S))^{0}$ is downward directed. Since, by (3), $(E\backslash (H,S))^{0}$ satisfies the strong CSP and $E\backslash (H,S)$ satisfies condition (L), we conclude, by Theorem \ref{comp.Irreducible}, that $I$ is completely irreducible, proving (1).
\end{proof} 

\begin{example} \label{Example where everyone comp.irr}
Let $E$ be the following graph.
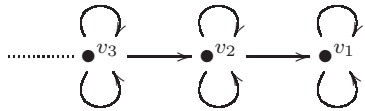
\vspace{0.2in}

\[
\xymatrix{\ar@{.}[r] & {\bullet}^{v_3} \ar [r] \ar@(ul,ur)\ar@(dl,dr) & {\bullet}^{v_2} \ar [r] \ar@(ul,ur)\ar@(dl,dr) & {\bullet}^{v_1} \ar@(ul,ur)\ar@(dl,dr)}
\]

\vspace{0.2in}

\noindent
Then $E$ is clearly row-finite and satisfies condition (K). Moreover, the proper hereditary
saturated subsets of $E^0$ are $H_{0}=\emptyset$ and $H_{i} =\{v_{1},\dots, v_{i}\}$, for all $i\geq1$. It follows that the admissible pairs $(H_i, \emptyset)$ form a chain under the partial order of admissible pairs, and $(E\backslash(H_i, \emptyset))^{0}$ satisfies the strong CSP for each $i \geq 0$. Hence, by Theorem~\ref{everyidealcompirred}, every ideal of $L_{K}(E)$ is completely irreducible.
\end{example}

Next we examine the Leavitt path algebras where all completely irreducible ideals are graded.

\begin{proposition} \label{Every comp.irred.graded}
The following are equivalent for any Leavitt path algebra $L=L_K(E)$.
\begin{enumerate}
\item[$(1)$] Every completely irreducible ideal of $L$ is graded.

\item[$(2)$] Every prime ideal of $L$ is graded.

\item[$(3)$] Every ideal of $L$ is graded.

\item[$(4)$] The graph $E$ satisfies condition (K).
\end{enumerate}
\end{proposition}

\begin{proof}
(1) $\Rightarrow$ (3). By Proposition~\ref{Everyone is int comp.irr.}, every proper ideal of $L$ is the intersection of completely irreducible ideals. Thus if all completely irreducible ideals are graded, then so are all the other ideals of $L$ (given that $L$ itself is graded).

(3) $\Rightarrow$ (2). This is a tautology.

(2) $\Rightarrow$ (1). By Theorem~\ref{comp.Irreducible}, every non-graded completely irreducible ideal is a power of a non-graded prime ideal. Thus, if all the prime ideals in $L$ are graded, then so are the completely irreducible ideals.

(3) $\Leftrightarrow$ (4). See \cite[Proposition 2.9.9]{AAS}.
\end{proof}

Let us now describe the Leavitt path algebras where every irreducible ideal is completely irreducible.

\begin{proposition} \label{Irredcible = Comp Irred} 
The following are equivalent for any Leavitt path algebra $L=L_K(E)$.
\begin{enumerate}
\item[$(1)$] Every irreducible ideal of $L$ is completely irreducible.

\item[$(2)$] Every graded prime ideal of $L$ is completely irreducible.

\item[$(3)$] Every ideal of $L$ is graded, and every prime ideal of $L$ is completely irreducible.

\item[$(4)$] The graph $E$ satisfies condition (K), and $(E\backslash (H,B_{H}))^{0}$ satisfies the strong CSP for each admissible pair $(H,B_{H})$ such that $E^{0} \setminus H$ is downward directed. 
\end{enumerate}
\end{proposition}

\begin{proof}
(1) $\Rightarrow$ (2). By Proposition~\ref{quasi-primaryequivalent}, every graded prime ideal of $L$ is irreducible, from which the desired conclusion follows.

(2) $\Rightarrow$ (3). Suppose that (2) holds and that $P$ is a non-graded prime ideal of $L$. Then, by Theorem~\ref{specificprime}, $P = I(H,B_{H})+\langle f(c)\rangle$, where $H=P \cap E^{0}$, $c$ is a cycle without (K), $E^0\setminus H = M(s(c))$, and $f(x) \in K[x, x^{-1}]$ is irreducible. By the same theorem, $I(H,B_{H})$ is a (graded) prime ideal, and so, by (2), $I(H,B_{H})$ is completely irreducible. Hence, by Theorem~\ref{comp.Irreducible}, $E\backslash (H,B_{H})$ satisfies condition (L), which contradicts $c$ being a cycle without (K) in $E$. Therefore, if (2) holds, then all prime ideals of $L$ must be graded, and so (3) must also hold, by Proposition~\ref{Every comp.irred.graded}.

(3) $\Rightarrow$ (4). Suppose that (3) holds. Then, by Proposition~\ref{Every comp.irred.graded}, $E$ satisfies condition $(K)$. Now let $H$ be a hereditary saturated subset of $E^0$, such that $E^{0} \setminus H$ is downward directed. Then $P=I(H,B_{H})$ is a graded prime ideal of $L$, by Theorem~\ref{specificprime}. Since $P$ is completely irreducible, by hypothesis, Theorem~\ref{comp.Irreducible} implies that $(E\backslash(H,B_{H}))^0$ satisfies the strong CSP.

(4) $\Rightarrow$ (1). Suppose that (4) holds. Let us take an irreducible ideal  $I$ of $L$, and show that it is completely irreducible. By Proposition~\ref{Every comp.irred.graded}, $I$ must be graded, and so Proposition~\ref{quasi-primaryequivalent} implies that $I$ is a graded prime ideal. By Remark~\ref{graded prime comp irred} and Theorem~\ref{specificprime}, we may assume that $I=I(H,B_{H})$, where $E^0 \setminus H = (E\backslash (H,B_{H}))^{0}$ is downward directed. Since, by hypothesis, $E$ satisfies condition (K), the quotient graph $E\backslash(H,B_{H})$ must satisfy condition (L). It now follows from Theorem~\ref{comp.Irreducible} that $I$ is completely irreducible.
\end{proof}

We conclude this section with an example of a graph $E$ such that every irreducible ideal of $L_K(E)$ is completely irreducible, but not every proper ideal is completely irreducible.

\begin{example} \label{div-eg}
Let $E$ be the following row-finite graph.
\[\xymatrix{ 
& \ar@{.}[l] & {\bullet}^{v_{-2}} \ar [l] & {\bullet}^{v_{-1}} \ar [l] & {\bullet}^{v_{0}} \ar [r] \ar [l] & {\bullet}^{v_{1}} \ar [r] & {\bullet}^{v_{2}} \ar [r] & \ar@{.}[r] & }
\]
Clearly $E$ satisfies condition (K), and it is easy to see that the hereditary saturated subsets of $E^0$ are precisely $E^0$, $W^+ = \{v_i \mid i \geq 1\}$, $W^- = \{v_i \mid i \leq -1\}$, and $\emptyset$. Now, each of $(E \setminus (W^+,\emptyset))^{0} = E^{0} \setminus W^+$ and $(E \setminus (W^-,\emptyset))^{0} = E^{0} \setminus W^-$
satisfies the strong CSP with respect to itself, while $E^{0} \setminus E^{0}$ and $E^{0} \setminus \emptyset$ are not downward directed. Thus, $E$ satisfies the properties in condition (4) of Proposition~\ref{Irredcible = Comp Irred}.

On the other hand, the hereditary saturated subsets of $E^0$ certainly do not form a chain, and so neither do the admissible pairs. Thus $E$ does not satisfy condition (3) in Theorem~\ref{everyidealcompirred}. We conclude that in $L_K(E)$ every irreducible ideal is completely irreducible, but not every proper ideal is completely irreducible. (Specifically, the zero ideal is not completely irreducible, by Proposition~\ref{Zero Comp Irreducible}.)
\end{example}

\section{Products and Intersections of Completely \\ Irreducible Ideals}\label{prod-comp-irred-section}

In this section we characterize the ideals in an arbitrary Leavitt path algebra that can be factored as products (and intersections) of completely irreducible ideals. We also describe the Leavitt path algebras in which every ideal can be represented as a product (and intersection) of such ideals. We shall require the following observation.

\begin{lemma}\label{gradedprodlemma} \cite[Lemma 4.3]{AMR}
The following are equivalent for any Leavitt path algebra $L=L_K(E)$ and positive integer $n$.
\begin{enumerate}
\item[$(1)$] The zero ideal is the (irredundant) intersection of $n$ prime ideals.

\item[$(2)$] The zero ideal is the (irredundant) intersection of $n$ graded prime ideals.

\item[$(3)$] $E^{0}$ is the (irredundant) union of $n$ maximal tails.
\end{enumerate}

Moreover, the maximal tails in (3) can be taken to be the complements in $E^0$ of the sets of vertices contained in the prime ideals in (1) or (2).
\end{lemma}

We are now ready for the last of our main results.

\begin{theorem} \label{Intersection of completely irreducibles}
Let $L=L_K(E)$ be a Leavitt path algebra, and let $I$ be a proper ideal of $L$. Then the following are equivalent.
\begin{enumerate}
\item[$(1)$] $I$ is the product of (finitely many) completely irreducible ideals.

\item[$(2)$] $I$ is the intersection of finitely many completely irreducible ideals.

\item[$(3)$] $I = I(H,S) + \sum_{i=1}^n \langle f_i(c_i)\rangle$, where 
\begin{enumerate}
\item[$(i)$] $0 \leq n$ is an integer, with $n=0$ indicating that $I=I(H,S)$; 
\item[$(ii)$] each $c_i$ is a cycle without exits in $E \setminus (H,S)$; 
\item[$(iii)$] each $f_i(x) \in K[x]$ is a polynomial with a nonzero constant term;
\item[$(iv)$] $(E \setminus (H,S))^{0} = \bigcup_{i=1}^m M_{i}$ is the irredundant union of finitely many maximal tails, with $n\leq m$, such that, for each $i \in \{1, \dots, n\}$ we have $s(c_i) \in M_{i}$ and $s(c_{i})\notin M_{j}$ for all $j\in \{1,\dots, m\} \setminus \{i\}$, and for each $i \in \{n+1, \dots, m\}$ every cycle with source in $M_i$ has an exit and $M_{i}$ satisfies the strong CSP.
\end{enumerate}
\end{enumerate}
\end{theorem}

\begin{proof}
(1) $\Rightarrow$ (3) Suppose that we have $I = P_1 \cdots P_m$ for some completely irreducible ideals $P_1, \dots, P_m$. By Theorem~\ref{arbitrary ideals}(1), $I=I(H,S)+\sum_{i\in Y} \langle f_{i}(c_{i})\rangle$, where each $c_i$ is a cycle without exits in $E \setminus (H,S)$, and each $f_i(x) \in K[x]$ is a polynomial with a nonzero constant term. Using Lemma~\ref{Product=Intersection}(2) and the fact that a product of graded ideals is graded, we have
\[
I(H,S) = \mathrm{gr}(I) \subseteq \bigcap_{i=1}^m \mathrm{gr}(P_i) = \prod_{i=1}^m \mathrm{gr}(P_i) \subseteq \mathrm{gr}(P_1 \cdots P_m) = \mathrm{gr}(I),
\]
and hence $I(H,S) = \bigcap_{i=1}^m \mathrm{gr}(P_i)$. By Theorem~\ref{comp.Irreducible}, Proposition~\ref{quasi-primaryequivalent}, and Theorem~\ref{specificprime}, each $\mathrm{gr}(P_i)$ is a prime ideal. Upon dropping terms, if needed, we may assume that the intersection is irredundant. If $I$ is non-graded, then, upon further reindexing, Lemma~\ref{unique-cycle-lemma}(2,4,5) gives that $|Y| = n \leq m$ for some positive integer $n$, that for each $i \in \{1, \dots, n\}$ we have $s(c_{i}) \notin P_{i}$ and $s(c_{i}) \in P_{j}$ for all $j \in \{1, \dots, m\} \setminus \{i\}$, and that 
\begin{equation*}
I = \prod_{i=1}^n (I(H_{i},B_{H_{i}}) + \langle f_{i}(c_{i})\rangle) \cdot \prod_{i=n+1}^m I(H_{i},B_{H_{i}}),  \tag{$\dag$}
\end{equation*}
where $\mathrm{gr}(P_{i}) = I(H_{i},B_{H_{i}})$ for each $i$. If $I$ is graded, then $I = I(H,S) = \bigcap_{i=1}^m \mathrm{gr}(P_i)$ and taking $n$ to be $0$, we have $I= \prod_{i=n+1}^m I(H_{i},B_{H_{i}})$, by Lemma~\ref{Product=Intersection}(2). Thus equation $(\dag)$ describes $I$ in both situations. Now, by Lemma~\ref{prime-gr-part}, $I(H_{i},B_{H_{i}}) + \langle f_{i}(c_{i})\rangle$ is a product of non-graded prime ideals, for each $i \in \{1, \dots, n\}$. Hence $(\dag)$ gives a representation of $I$ a product of prime-power ideals. Upon combining copies of the same prime ideal, dropping any redundant terms, and reindexing, we can apply Theorem~\ref{Uniqueness} (and Theorem~\ref{comp.Irreducible}) to conclude that for each $i \in \{n+1, \dots, m\}$ the ideal $I(H_{i},B_{H_{i}})$ must be one of the completely irreducible ideals in the product $I = P_1 \cdots P_m$.

Since $I(H,S) = \bigcap_{i=1}^m \mathrm{gr}(P_i)$, in the ring $L/I(H,S)\cong L_{K}(E\backslash (H,S))$, by Lemma \ref{gradedprodlemma}, we have $(E\backslash (H,S))^{0} = \bigcup_{i=1}^m M_i$, where each
\[
M_i = (E\backslash (H,S))^{0} \setminus ((E\backslash (H,S))^{0} \cap \bar{Q}_i) 
\]
is a maximal tail, the union is irredundant, and $\bar{Q}_i = \mathrm{gr}(P_{i})/I(H,S)$. Then, for each $i \in \{1, \dots, n\}$, we have $s(c_i) \in M_{i}$ and $s(c_{i})\notin M_{j}$ for all $j\in \{1,\cdots, m\} \setminus \{i\}$. Moreover, for each $i \in \{n+1, \dots, m\}$, the ideal $\bar{Q}_i = I(H_{i},B_{H_{i}}) /I(H,S)$ is completely irreducible, and hence, by Theorem~\ref{comp.Irreducible}, every cycle with source in $M_i$ has an exit and $M_{i}$ satisfies the strong CSP. Thus the conditions in (3) hold for $I$.

(3) $\Rightarrow$ (1). Supposing that $I$ satisfies (3), we have $(E\backslash (H,S))^{0}= \bigcup_{i=1}^{m} M_{i}$, with the $M_i$ as in $(iv)$. In $L/I(H,S)\cong L_{K}(E\backslash (H,S))$, let $H_{i} = (E\backslash (H,S))^{0} \backslash M_{i}$, and let $P_i$ be the ideal of $L$ containing $I(H,S)$ such that $\bar{P}_{i} = P_i/I(H,S) = I(H_{i},B_{H_{i}})$, for each $i \in \{1, \dots, m\}$. Then, by hypothesis, $c_{i}\notin \bar{P}_{i}$ for $i \in \{1,\dots, n\}$, and, by Theorem~\ref{specificprime}, each $\bar{P}_{i}$ is a (graded) prime ideal. Moreover, by Theorem~\ref{comp.Irreducible}, $\bar{P}_{i}$ is completely irreducible for $i \in \{n+1,\dots, m\}$. Also, by Lemma~\ref{gradedprodlemma}, $\{0\} = \bigcap_{i=1}^m \bar{P}_{i}$ is an irredundant intersection in $L/I\cong L_{K}(E\backslash (H,S))$. Then, using Lemma~\ref{Product=Intersection}(2) and Lemma~\ref{unique-cycle-lemma} (depending on whether or not $I$ is graded), we have
\begin{equation*}
I/I(H,S) = \prod_{i=1}^n (I(H_{i},B_{H_{i}}) + \langle f_{i}(c_{i})\rangle) \cdot \prod_{i=n+1}^m I(H_{i},B_{H_{i}}). \tag{$\ddag$}
\end{equation*}
Now, for each $i \in \{1, \dots, n\}$, let $J_i = P_i + \langle f_{i}(c_{i})\rangle \subseteq L$. Then, by Lemma~\ref{prime-gr-part}, each $\bar{J}_i$ is a product of non-graded prime ideals, and is hence is a product of completely irreducible ideals, by Theorem~\ref{comp.Irreducible}. Then equation $(\ddag)$ implies that $I = J_1 \cdots J_n P_{n+1} \cdots P_m$, which shows that $I$ is a product of completely irreducible ideals.

The equivalence of (1) and (2) follows from Proposition~\ref{prime-power-intersect} and Theorem~\ref{comp.Irreducible}.
\end{proof}

This theorem makes it easy to construct Leavitt path algebras having ideals that are or are not intersections (and products) of finitely many completely irreducible ideals. 

\begin{example} \label{No intersect of finitely many Comp Irred} 
Let $E$ be the following graph.

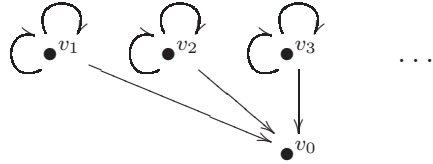
\vspace{0.2in}

\[
\xymatrix{{\bullet}^{v_1}\ar@(dl,ul) \ar@(ul,ur) \ar[drr] & {\bullet}^{v_2} \ar@(dl,ul) \ar@(ul,ur) \ar[dr] & {\bullet}^{v_3} \ar@(dl,ul) \ar@(ul,ur) \ar[d] & \dots   \\
 &  &  {\bullet}^{v_0} & } 
\]

\vspace{0.2in}

\noindent
Clearly, $\{v_0\}$ is hereditary and saturated. Taking $I = I(\{v_0\}, \emptyset)$, we shall show that $I$ is not the intersection or product of any finite collection of completely irreducible ideals in $L_K(E)$.

It is easy to see that $E\backslash (\{v_0\}, \emptyset)$ has the following form.

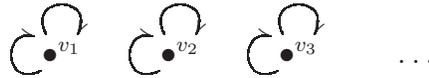
\vspace{0.2in}

\[
\xymatrix{{\bullet}^{v_1}\ar@(dl,ul) \ar@(ul,ur) & {\bullet}^{v_2} \ar@(dl,ul) \ar@(ul,ur) & {\bullet}^{v_3} \ar@(dl,ul) \ar@(ul,ur) & \dots} 
\]

\vspace{0.2in}

\noindent
In this graph, the maximal tails are precisely the sets of the form $\{v_i\}$ ($i \geq 1$), and so $(E\backslash (\{v_0\}, \emptyset))^0$ cannot be expressed as the union of finitely many maximal tails. Therefore $I$ is not the intersection or product of any finite collection of completely irreducible ideals in $L_K(E)$, by Theorem~\ref{Intersection of completely irreducibles}.

Now let $E$ to be the following graph.

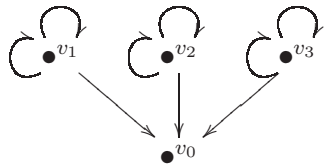
\vspace{0.2in}

\[
\xymatrix{{\bullet}^{v_1}\ar@(dl,ul) \ar@(ul,ur) \ar[dr] & {\bullet}^{v_2} \ar@(dl,ul) \ar@(ul,ur) \ar[d] & {\bullet}^{v_3} \ar@(dl,ul) \ar@(ul,ur) \ar[dl] \\
 &  {\bullet}^{v_0} & } 
\]

\vspace{0.2in}

\noindent
Then very similar reasoning shows that in $L_K(E)$ the ideal $I = I(\{v_0\}, \emptyset)$ is a product and intersection of finitely many completely irreducible ideals. 
\end{example}

Our final goal is to classify the Leavitt path algebras in which every proper ideal is a product (and intersection) of finitely many completely irreducible ideals.

\begin{theorem} \label{all-prod-irred} 
The following are equivalent for any Leavitt path algebra $L=L_{K}(E)$.
\begin{enumerate}
\item[$(1)$] Every proper ideal of $L$ is the product of (finitely many) completely irreducible ideals.

\item[$(2)$] Every proper ideal of $L$ is the intersection of finitely many completely irreducible ideals.

\item[$(3)$] The graph $E$ satisfies condition (K), and for every admissible pair $(H,S)$ with $H \neq E^{0}$, $(E \setminus (H,S))^{0}$ is the union of finitely many maximal tails, each satisfying the strong CSP.
\end{enumerate}
\end{theorem}

\begin{proof}
(1) $\Leftrightarrow$ (2). This follows from Theorem~\ref{Intersection of completely irreducibles}.

(2) $\Rightarrow$ (3). Suppose that (2) holds. Then, in particular, every irreducible ideal is the intersection of finitely many completely irreducible ideals, and so must actually be completely irreducible. Thus, by Proposition~\ref{Irredcible = Comp Irred}, the graph E satisfies condition (K).

Now, let $(H,S)$ be an admissible pair, with $H \neq E^{0}$. Then, by hypothesis, $I(H,S)$ is the intersection of finitely many completely irreducible ideals. Thus, by Theorem~\ref{Intersection of completely irreducibles}, $(E \setminus (H,S))^{0}$ is the union of finitely many maximal tails, each satisfying the strong CSP.

(3) $\Rightarrow$ (2). Suppose that (3) holds. Since the graph $E$ satisfies condition (K), every ideal of $L$ is graded, by Proposition~\ref{Every comp.irred.graded}. Now, let $I=I(H,S)$ be an arbitrary proper (graded) ideal of $L$. Since $E$ satisfies condition $(K)$, every cycle in $E \setminus (H,S)$ has an exit. Thus, by Theorem~\ref{Intersection of completely irreducibles}, (3) implies that $I$ is the intersection of finitely many completely irreducible ideals, giving the desired conclusion.
\end{proof}

For graphs with finitely many vertices, all the clauses in condition (3) of Theorem~\ref{all-prod-irred}, except condition (K), are satisfied automatically, giving us the following result.

\begin{corollary}
Let $E$ be a graph such that $E^0$ is finite. Then every proper ideal of $L_K(E)$ is a product of completely irreducible ideals if and only if $E$ satisfies condition (K).
\end{corollary}

Let us further illustrate Theorem~\ref{all-prod-irred} by giving an example of a Leavitt path algebra where every proper ideal is a product of completely irreducible ideals, but may not be completely irreducible itself.

\begin{example}
Let $E$ be the following graph (see Example~\ref{div-eg}).
\[\xymatrix{ 
& \ar@{.}[l] & {\bullet}^{v_{-2}} \ar [l] & {\bullet}^{v_{-1}} \ar [l] & {\bullet}^{v_{0}} \ar [r] \ar [l] & {\bullet}^{v_{1}} \ar [r] & {\bullet}^{v_{2}} \ar [r] & \ar@{.}[r] & }
\]
Then $E$ satisfies condition (K), and the hereditary saturated subsets of $E^0$ are precisely $E^0$, $W^+ = \{v_i \mid i \geq 1\}$, $W^- = \{v_i \mid i \leq -1\}$, and $\emptyset$. Now,
\[
(E \setminus (W^+,\emptyset))^{0} = E^0\setminus W^+ = W^-\cup \{v_0\} \text{ and } (E \setminus (W^-,\emptyset))^{0} = E^0\setminus W^- = W^+\cup \{v_0\}
\] 
are maximal tails, each satisfying the strong CSP with respect to itself, and 
\[
(E \setminus (\emptyset,\emptyset))^{0} = E^0 = (W^-\cup \{v_0\}) \cup (W^+\cup \{v_0\}).
\] 
It follows that $E$ satisfies condition (3) in Theorem~\ref{all-prod-irred}. On the other hand, as noted in Example~\ref{div-eg}, $E$ does not satisfy condition (3) in Theorem~\ref{everyidealcompirred}. Thus in $L_K(E)$ every proper ideal is a product of completely irreducible ideals, but not every proper ideal is completely irreducible. 
\end{example}

\subsection*{Acknowledgement}

We are grateful to the referee for a careful reading of the manuscript.

\vspace{.1in}

\noindent
Department of Mathematics, University of Colorado, Colorado Springs, CO, 80918, USA \newline
\noindent 
{\href{mailto:zmesyan@uccs.edu}{zmesyan@uccs.edu}},  {\href{mailto:krangasw@uccs.edu}{krangasw@uccs.edu}}

\end{document}